\documentclass[11pt]{amsart}
\usepackage{hyperref}
\usepackage{amssymb,amsmath}
\usepackage{amsthm}
\usepackage{latexsym}

\numberwithin{equation}{section}
\newtheorem{thm}[equation]{Theorem}
\newtheorem{lem}[equation]{Lemma}
\newtheorem{prop}[equation]{Proposition}

\theoremstyle{remark}
\newtheorem{remark}[equation]{Remark}
\newtheorem{definition}[equation]{Definition}

\newcommand{\abs}[1]{\lvert #1 \rvert}
\newcommand{\labs}[1]{\left| #1 \right|}

\newcommand{\norml}[2][2]{\abs{#2}_{#1}}

\newcommand{\scalarh}[2]{\langle #1 \,|\, #2 \rangle}
\newcommand{\normh}[1]{\lVert #1 \rVert}
\newcommand{\settc}[2]{\bigl\{\,#1 \bigm\vert #2\,\bigr\}}

\newcommand{\fp}{\hat{f}^{+}}
\newcommand{\fm}{\hat{f}^{-}}
\newcommand{\fpm}{\hat{f}^{\pm}}

\newcommand{\ffp}{F^{+}}
\newcommand{\ffpm}{F^{\pm}}

\newcommand{\Fp}{\mathcal{F}^{+}_{\lambda}}

\newcommand{\Fpm}{\mathcal{F}^{\pm}_{\lambda}}

\newcommand{\Ip}{I^{+}_{\lambda}}
\newcommand{\Imm}{I^{-}_{\lambda}}
\newcommand{\Ipm}{I^{\pm}_{\lambda}}

\newcommand{\vdp}{v_{\delta}}
\newcommand{\vdb}{\bar{v}_{\delta}}

\newcommand{\gp}[1]{\hat{g}^{+}_{#1}}
\newcommand{\gm}[1]{\hat{g}^{-}_{#1}}

\DeclareMathOperator{\dive}{div}

\newcounter{tesi}

\newenvironment{ipotesi}[1]%
	{%
		\begin{list}%
			{\textbf{\upshape{(#1\arabic{tesi})}}}%
			{%
				\usecounter{tesi}%
				\renewcommand{\thetesi}{\textbf{\upshape{(#1\arabic{tesi})}}}%
				\settowidth{\labelwidth}{\thetesi}%
				\setlength{\labelsep}{5pt}%
				\settowidth{\itemindent}{\thetesi\hskip\labelsep}%
				\setlength{\leftmargin}{0pt}%
			}%
	}%
	{\end{list}%
	           }%

\title[radial solutions for mean curvature operator]{Existence of
multiple radial solutions for nonlinear equation involving the mean
curvature operator in Lorentz-Minkowski space}

\author{Vittorio Coti Zelati}
\author{Xu Dong}
\author{Yuanhong Wei}



\begin{document}

\begin{abstract}
	This paper devotes to the multiplicity of the radial solutions for
	the Dirichlet problem for nonlinear equations involving the mean
	curvature operator in Lorentz-Minkowski space
	\begin{equation*}
		-\dive (\frac{\nabla u}{\sqrt{1 - \abs{\nabla u}^{2}}}) =
		\lambda b(\abs{x})\abs{u}^{q-2}u + f(\abs{x}, u) \qquad
		\text{in } B_{R},
	\end{equation*}
	in case $q \in (1,2)$ and $f(\abs{x}, s)$ is superlinear in $s$.

	Solutions are found using Szulkin's critical point theory for
	non-smooth functional. Multiplicity results are also given for
	some cases in which $f$ depends also on (the absolute value of)
	the gradient of $u$.
\end{abstract}

\maketitle

\section{Introduction}

This paper is about existence of multiple radial solutions for
nonlinear equations involving the mean curvature operator in
Lorentz-Minkowski space
\begin{equation}
	\label{MCE}
	\begin{cases}
		-\dive (\frac{\nabla u}{\sqrt{1 - \abs{\nabla u}^{2}}}) =
		\lambda b(\abs{x})\abs{u}^{q-2}u + f(\abs{x}, u) & \text{in }
		B_{R}, \\
		u = 0& \text{on }\partial B_{R}
	\end{cases}
\end{equation}
where $B_{R}$ is the open ball centered at $0$ of radius $R$ in the
Euclidean space $\mathbb{R}^{N}$, $N \geq 3$, $\lambda > 0$ is a real
parameter, $1 < q < 2$. Here $\abs{\cdot}$ stands for the Euclidian
norm in $\mathbb{R}^{N}$.

Under this ansatz $u(\abs{x})$, where $u \colon [0,+\infty) \to
\mathbb{R}$ and letting $r = \abs{x}$ the equation \eqref{MCE} reduces
to the one-dimensional mixed boundary problem
\begin{equation}
	\label{RMCE}
	\tag{BP}
	\begin{cases}
		-\left( \frac{r^{N-1} u'}{\sqrt{1 - \abs{u'}^{2}}}
		\right)' = r^{N-1} \lambda b(r) \abs{u}^{q-2} u + r^{N-1}
		f(r,u) &r \in (0,R), \\
		u'(0) = u(R) = 0 & 
	\end{cases}.
\end{equation}

The operator $\mathcal{Q} \colon \mathbb{R}^{N} \to \mathbb{R}$,
\begin{equation*}
	\mathcal{Q}(u):= - \dive(\frac{\nabla u}{\sqrt{1-\abs{\nabla
	u}^{2}}}),
\end{equation*}
is related to the mean curvature of space-like hypersurfaces in
Lorentz-Minkowski space and to some problems in Physics, see for
example in \cite{Cheng_Yau_1976, Flaherty_1979} for application in
geometry and \cite{Bonheure_dAvenia_Pomponio_2016} where its relation
to the Born-Infeld theory is explained.

The study of the nonlinear equations involving the operator
$\mathcal{Q}$ has received a lot of attention after the publication of
the paper \cite{Bereanu_Jebelean_Mawhin_2009} by Bereanu, Jebelean,
Mawhin and since then several different methods have been used to
tackle different nonlinear problems involving such and operator.

Fixed point theorems have been used to prove existence of radial
solutions for both Dirichlet and Neumann problems with different kind
of nonlinearities in \cite{Bereanu_Jebelean_Mawhin_2009,
Bereanu_Jebelean_Mawhin_2010, Bereanu_Jebelean_Torres_2013,
Bereanu_Jebelean_Torres_2013a}.

The same authors have then introduced a variational approach to the
problem to prove, using critical point theory for nonsmooth
functionals the existence radial solutions for the Neumann problem in
\cite{Bereanu_Jebelean_Mawhin_2011a, Bereanu_Jebelean_Mawhin_2013}. In
\cite{Bereanu_Jebelean_Mawhin_2014} they prove, with similar
techniques, existence of (not radial) solutions for the Dirichlet
problem.

Also using critical point theory for nonsmooth functionals Bonheure,
d'Avenia and Pomponio in \cite{Bonheure_dAvenia_Pomponio_2016} prove
existence and uniqueness of the weak solution for the electrostatic
Born-Infeld equation
\begin{equation*}
	\begin{cases}
		- \dive(\frac{\nabla u}{\sqrt{1-\abs{\nabla u}^{2}}}) = \rho&
		\text{in } \mathbb{R}^{N} \\
		\lim_{\abs{x} \to +\infty} u(x) = 0
	\end{cases}
\end{equation*}
where $\rho$ is an assigned extended charge density.

Motivated in particular by the results of Coelho, Corsato, and
Rivetti, who in the paper \cite{Coelho_Corsato_Rivetti_2014} prove
existence of one, two or three positive radial solutions with a
nonlinearity similar to that in \eqref{MCE}, we prove here for that
equation some existence and multiplicity of radial solutions. Due to
the lack of regularity of the operator $\mathcal{Q}$, we are lead to
study existence of critical points for a non-smooth functional
following Szulkin \cite{Szulkin_1986} in a way related to what has
been done in \cite{Bereanu_Jebelean_Mawhin_2014}. This approach is
different from the one used by Coelho, Corsato, and Rivetti in
\cite{Coelho_Corsato_Rivetti_2014}, where an equivalent, non-singular
problem is introduced and studied.

In our Theorem \ref{thm:ex6} we show existence of 3 positive and 3
negative solutions for equation \eqref{MCE}. This result is similar to
the one in \cite{Coelho_Corsato_Rivetti_2014} (from which one can
easily deduce also existence of negative solutions), although with a
(slightly) different proof (for a more precise comparison of the
results see Remark \ref{rem:coelhoetc}). Our main contributions are
the existence of a seventh solution, see Theorem \ref{thm:ex7},
applying in this situation an idea of Ambrosetti, Garcia Azorero and
Peral in \cite{Ambrosetti_GarciaAzorero_Peral_1996}, and the existence
of multiple solutions in case the nonlinear term depends also from the
norm of the gradient of the solution, see Theorem \ref{thm:thg}.

Let $F(r,s) = \int_{0}^{s} f(r,t) \, dt$. On $f$ and $b$ we will
assume that
\begin{ipotesi}{F}
	\item\label{it:fC} $f \in C([0, R] \times [-R, R],
	\mathbb{R})$;

	\item\label{it:fat0} $\lim_{s \to 0} \frac{f(r,s)}{s} = 0$
	uniformly for $r \in [0, R]$;

	\item\label{it:AR}
	$sf (r, s) > 0$ for every $0 < \abs{s} \leq R$ and $r \in [0, R]$;

	\item\label{it:fsuper} $F(r,s) \geq a_{1}\abs{s}^{\theta} - a_{2}$
	for all $(r,s) \in [0,R] \times [-R, R]$ for some $\theta > 2$ and
	$a_{1} > 0$, $a_{2} \geq 0$ such that
	\begin{equation}
		\label{eq:stimaaa}
		(1 + a_{2}) \frac{1}{N} - a_{1} R^{\theta}
		\frac{\Gamma(N)\Gamma(\theta+1)}{\Gamma(N+\theta+1)} \leq -1
	\end{equation}
	(here $\Gamma$ denotes the usual Gamma funtion);
	
	\item\label{it:asymmetry} there is $\sigma > 0$ and $\bar{\theta} 
	> 4$ such that $\abs{F(r, -s) - F(r, s)} \leq
	\abs{s}^{\bar{\theta}}$ for all $(r,s) \in [0,R] \times
	[-\sigma,\sigma]$;
	
	\item\label{it:fprimo} $f \in C^{1}([0, R] \times [-R, R],
	\mathbb{R})$ and there is $\sigma > 0$ and $\tilde{\theta} > 2$
	such that $\abs{f'(r, s)} \leq c \abs{s}^{\tilde{\theta}-2}$ for
	all $(r,s) \in [0,R] \times [-\sigma,\sigma]$,
\end{ipotesi}
and 
\begin{ipotesi}{B}
	\item\label{it:b} $b \in C([0,R], \mathbb{R})$ is such that $0 <
	b_{0} \leq b(r) \leq b_{1}$ for some real constants $b_{0}$ and
	$b_{1}$. 
\end{ipotesi}

\begin{remark}
	The assumption \ref{it:fsuper} is satisfied for all $R$ large
	enough if $F(r,s) \geq a_{1}\abs{s}^{\theta} - a_{2}$ for all
	$(r,s) \in [0,+\infty) \times \mathbb{R}$. It implies that the
	nonlinear term $f(r,s)$ is superlinear. It can also be satisfied
	taking if $f$ is multiplied by a large constant $\mu$ as in the
	paper \cite{Coelho_Corsato_Rivetti_2014}, see remark
	\ref{rem:coelhoetc}.
	
	The assumption \ref{it:asymmetry} tells us that the function $F$ 
	is close to be an even function in $s$ when $s$ is small. 
	
	The assumption \ref{it:fprimo} implies \ref{it:fat0}, and we can 
	assume $\tilde{\theta} \in (2, \frac{2N}{N-2})$.
	
	From the validity of the well known Ambrosetti-Rabinowitz
	condition $0 < \theta F(r,s) \leq sf(r, s)$ for all $r \in
	[0,+\infty)$ and $s \in \mathbb{R} \setminus \{0\}$ follows that
	\ref{it:AR} and \ref{it:fsuper} hold provided $R$ is large enough.
	From such an assumption also follows that $F(r,s) \leq
	c\abs{s}^{\theta}$ for $\abs{s} \leq 1$ and $r \in [0,R]$ .
\end{remark}

\begin{thm}
	\label{thm:ex6}
	Let $f$ satisfy \ref{it:fC}, \ref{it:fat0}, \ref{it:AR},
	\ref{it:fsuper} and assume \ref{it:b} holds. Then 
	\begin{itemize}
		\item For $\lambda > 0$ \eqref{RMCE} has one positive solution
		$u_{+}$ and one negative solution $u_{-}$;
	
		\item There is $\lambda_{*}(R)$ such that for all $\lambda \in
		(0, \lambda_{*}(R)) \eqref{RMCE}$ has one additional positive
		solution $w_{+}$ and one additional negative solution $w_{-}$
		of mountain pass type;
	
		\item There is $\lambda_{**} \in (0, \lambda_{*}(R))$ such
		that for all $\lambda \in (0, \lambda_{**}) \eqref{RMCE}$ has
		third positive solution $v_{+} $ and third negative solution
		$v_{-}$.
	\end{itemize}
\end{thm}

\begin{remark}
	\label{rem:coelhoetc}
	This Theorem should be compared with the result by Coelho, Corsato
	and Rivetti in \cite[Theorem 3.1]{Coelho_Corsato_Rivetti_2014}. 
	In that Theorem the authors consider  a class of equations of the form 
	\begin{equation*}
		\begin{cases}
			-\left( \frac{r^{N-1} u'}{\sqrt{1 - \abs{u'}^{2}}}
			\right)' = r^{N-1} \lambda g(r,u) + \mu r^{N-1} f(r,u) &r
			\in (0,R), \\
			u'(0) = u(R) = 0 
		\end{cases}.
	\end{equation*}
	Their assumptions on $g(r,u)$ and $f(r,u)$ are weaker than our
	assumptions on the corresponding nonlinear terms. They prove
	existence of $3$ positive solutions provided $\mu$ is large enough
	and $\lambda < \lambda_{*}(\mu)$. The method of the proof there
	relies on the introduction of a modified problem that allows them
	to use standard critical point theory in Hilbert space for $C^{1}$
	functional in order to find the three critical point (a global 
	minimum, a local minimum and a mountain pass critical point).
\end{remark}

Our main improvement on the paper \cite{Coelho_Corsato_Rivetti_2014} 
is the existence of a seventh solution, a mountain pass critical 
point between the two local minima $v_{\pm}$.
\begin{thm}
	\label{thm:ex7}
	Let $f$ satisfy \ref{it:fC}, \ref{it:fat0}, \ref{it:AR},
	\ref{it:fsuper}, \ref{it:asymmetry} and \ref{it:fprimo}, and
	assume \ref{it:b} holds. Then there exists a $\lambda_{***} \in
	(0, \lambda_{**})$ such that for all
	$\lambda \in (0, \lambda_{***})$ \eqref{RMCE} has an additional
	solution $v$ of mountain pass type such that $I_{\lambda}(v) < 0$.
\end{thm}

\begin{remark}
	An example of a function satisfying assumptions
	\ref{it:fC}--\ref{it:fprimo} is $f(r,s) =
	\frac{a_{1}}{5}\abs{s}^{3}s^{+} - \frac{a_{2}}{5}\abs{s}^{3}s^{-}$
	provided $a_{1}$ and $a_{2}$ are such that
	\begin{equation*}
		\frac{120a_{i} R^{5}}{(N+5)\cdots N} \geq 1 + \frac{1}{N}.
	\end{equation*}
	If $N = 3$ it is satisfied if
	\begin{equation*}
		\frac{a_{i}R^{5}}{168} \geq \frac{4}{3}.
	\end{equation*}
\end{remark}

We also consider here the mean curvature equation in
Lorentz-Minkowski space which with a gradient term.
\begin{equation}
	\label{GMCE}
	\begin{cases}
		-\dive (\frac{\nabla u}{\sqrt{1 - \abs{\nabla u}^{2}}}) =
		\lambda b(\abs{x}) \abs{u}^{q-2}u + g(\abs{x}, u, \abs{\nabla
		u} ) &\text{in } B_{R}, \\
		u=0 &\text{on } \partial B_{R}(0)
	\end{cases}
\end{equation}
We assume that the nonlinear term $g$ satisfies the following
assumptions
\begin{ipotesi}{G}
	\item\label{it:fC*} $g \in C([0, R] \times [-R,R] \times [-1,1],
	\mathbb{R})$ ;
	
	\item\label{it:fat0*}  $\lim_{s \to 0} \frac{g(r,s,\xi)}{s} = 0$
	uniformly for  $r \in [0,R]$ and $\xi \in [-1,1]$;
	
	\item\label{it:AR*}
	$sg(r, s, \xi) > 0$ for every $0 < \abs{s} \leq R$, $r \in [0, 
	R]$ and $\xi \in [-1,1]$;

	\item\label{it:fsuper*} $G(r,s, \xi) \geq a_{1}\abs{s}^{\theta} -
	a_{2}$ for all $(r,s, \xi) \in [0,R] \times [ -R , R] \times [-1, 1]$
	for some $\theta > 2$ and $a_{1} > 0$, $a_{2} \geq 0$ such that
	\begin{equation}
		\label{eq:stimaaa*}
		(1 + a_{2}) \frac{1}{N} - a_{1} R^{\theta}
		\frac{\Gamma(N)\Gamma(\theta+1)}{\Gamma(N+\theta+1)} \leq -1;
	\end{equation}
	
	\item\label{it:lip} there exist positive constants $L_{1}$,
	$L_{2}$ such that
	\begin{align*}
		&\abs{g(r, s_{1}, \xi) - g(r, s_{2}, \xi)} \leq
		L_{1}\abs{s_{1} - s_{2}}\\
		&\abs{g(r, s, \xi_{1}) - g(r, s, \xi_{2})} \leq L_{2}
		\abs{\xi_{1} - \xi_{2}}
	\end{align*}
	for any $r\in [0, R]$, $\abs{s} \leq R$, $\abs{\xi} \leq 1$ 
	where
	\begin{equation}
		\label{eq:lipconst}
		L_{1} C_{2} < \frac{1}{4} \qquad \hbox{and} \qquad L_{2}
		\sqrt{C_{2}} < \frac{1}{2}.
	\end{equation}
	(Here $C_{2} = C(N,2,R)$ is the Sobolev's embedding constant, 
	see Lemma \ref{estimate3} below). 
\end{ipotesi}

Our main result about the mean curvature equation in Lorentz-Minkowski
space with gradient term reads as follows.	

\begin{thm}
	\label{thm:thg}
	Suppose that $g$ satisfies \ref{it:fC*}, \ref{it:fat0*},
	\ref{it:AR*}, \ref{it:fsuper*} and \ref{it:lip}, and $b$ satisfies
	\ref{it:b}. Then there exists $\bar{\lambda}$ such that, for any
	$\lambda \in (0, \bar{\lambda})$ problem \eqref{GMCE} has at least
	two positive and two negative nontrivial solutions.
\end{thm}

\section{ Preliminaries}
\label{sec:prel}
In this section, we focus on some preliminaries and introduce some
technical tools that will be used in the sequel. 

\subsection{Szulkin critical point theory for non-smooth functionals}

To find solutions of our problem we will use the the critical point
theory for non-smooth functionals developed by Szulkin in
\cite{Szulkin_1986}, which we present briefly here.

\begin{definition}
	\label{def:d1}
	Let $X$ be a real Banach space and $\Psi \colon X \to
	(-\infty,+\infty]$ be a convex lower semi-continuous function. Let
	$D(\Psi) = \settc{u \in X}{\Psi(u)<+\infty}$ be the effective
	domain of $\Psi$. Denote by $X^{*}$ the dual of $X$ and by
	$\scalarh{\cdot}{\cdot}$ the duality pairing between $X^{*}$ and
	$X$. For $u\in D(\Psi)$, the set
	\begin{equation*}
		\partial\Psi(u) = \settc{u^{*} \in X^{*}}{\Psi(v) - \Psi(u)
		\geq \scalarh{u^{*}}{v - u} \ \forall v \in D(\Psi)}
	\end{equation*}
	is called the sub-differential of $\Psi$ at $u$. 
	
	For a functional of the form $I(u) = \Psi(u) + \mathcal{F}(u)$,
	with $\Psi(u)$ as above and $\mathcal{F} \in C^{1}(X,\mathbb{R})$,
	then $u \in D(\Psi)$ is said to be a critical point for $I$ if
	\begin{equation}
		\label{def:criticalpoint}
		\scalarh{\mathcal{F}'(u)}{v-u} + \Psi(v) - \Psi(u) \geq 0,
		\qquad \forall v\in D(\Psi).
	\end{equation}
	or, equivalently, if $-\mathcal{F}'(u) \in \partial\Psi(u)$. 
	
	A number $c \in \mathbb{R}$ such that $I^{-1}(c)$ contains a
	critical point is called a critical value of $I$.
\end{definition}

\begin{prop}[Proposition 1.1 in \cite{Szulkin_1986}]
	If $u \in X$ is a local minimum for $I$ then $u$ is a critical 
	point.
\end{prop}

Always following \cite{Szulkin_1986} we now state in this setting the
Palais-Smale compactness condition.
\begin{definition}
	The functional $I(u) = \Psi(u) + \mathcal{F}(u)$ satisfies the
	Palais-Smale condition if every sequence $\{u_{n}\} \in X$ for
	which $I(u_{n}) \to c$ and
	\begin{equation}
		\label{eq:ps}
		\scalarh{\mathcal{F}'(u_{n})}{v - u_{n}} + \Psi(v) -
		\Psi(u_{n}) \geq-\varepsilon_{n} \normh{v-u_{n}} \qquad
		\forall v\in D(\Psi),
	\end{equation}
	where $\varepsilon_{n} \to 0$ possesses a convergent subsequence.
\end{definition}

\begin{prop}[Theorem 1.7 of \cite{Szulkin_1986}]
	\label{minimumcritical}
	If $I$ is bounded from below and satisfies (PS) then 
	\begin{equation*}
		c = \inf_{X} I
	\end{equation*}
	is a critical value.
\end{prop}

\begin{prop}(Mountain Pass Theorem, see Theorem 3.2 of
\cite{Szulkin_1986})
	\label{thm:szulkin}
	Suppose that $I = \Psi + \mathcal{F}$, where $\mathcal{F} \in
	C^{1}(X, \mathbb{R})$ and $\Psi \colon X \to (-\infty, +\infty]$
	is convex, proper (i.e. $\Psi \not\equiv +\infty) $, and lower
	semi-continuous, is a function satisfying Palais-Smale condition
	and
	\begin{enumerate}
		\item[(i)] $I(0)=0$ and there exist $M$, $\rho > 0$ such that
		$I|_{\partial B_{\rho}}\geq M$,
		
		\item[(ii)] $I(e) \leq 0$ for some $e \notin \overline{B}_{\rho}$.
	\end{enumerate}
	Then $I$ has a critical value $$c \geq M$$ which may be
	characterized by
	\begin{equation*}
		c = \inf_{\phi \in \Phi} \sup_{s \in [0, 1]} I(\phi(s)),
	\end{equation*}
	where $\Phi = \settc{\phi \in C ([0, 1] , X)}{\phi(0) = 0, \ 
	\phi(1) = e}$.
\end{prop}

\subsection{The functional space and some inequalities}
We let
\begin{equation*}
	\mathcal{A}(0,R) = \settc{u \in W_{\text{loc}}^{1,1}(0,
	R]}{\int^{R}_{0} r^{N-1} \abs{u'}^{2} \, dr < +\infty, \ u(R)=0},
\end{equation*}
which is an Hilbert space endowed with the scalar product
\begin{equation*}
	\scalarh{u}{v} := \int^{R}_{0} r^{N-1} u'(r) v'(r) \, dr
\end{equation*}
and the corresponding norm
\begin{equation*}
	\normh{u}^{2} = \scalarh{u}{u}.
\end{equation*}
For $p \geq 1$, we consider the following weighted Lebesgue spaces
\begin{equation*}
	L^{p}_{N-1}(0,R) := \settc{ u \colon (0,R) \to \mathbb{R},
	\text{measurable}}{ \int^{R}_{0} r^{N-1} \abs{u}^{p}\, dr <
	+\infty}.
\end{equation*}
The space $L^{p}_{N-1}(0,R)$ ($L^{p}_{N-1}$ for brevity) is a Banach
space with norm
\begin{equation*}
	\norml[N-1,p]{u} = \left( \int^{R}_{0} r^{N-1} \abs{u}^{p} \, dr
	\right)^{\frac{1}{p}},
\end{equation*}
while with $L^{\infty}$ we denote the usual Lebesgue space
$L^{\infty}(0,R)$ and with $\norml[\infty]{ \cdot }$ its norm. Let us
remark that for all radially symmetric functions $v \colon
\mathbb{R}^{N} \to \mathbb{R}$ we have that $v(x) = u(\abs{x})$ for
some function $u \colon [0,+\infty) \to \mathbb{R}$, and
\begin{align*}
	&\int_{B(0,R)} \abs{\nabla v(x)}^{2} \, dx = \omega_{N-1}
	\int_{0}^{R} r^{N-1} \abs{u'(r)}^{2} \, dr, \\
	&\int_{B(0,R)} \abs{v(x)}^{p} \, dx = \omega_{N-1} \int_{0}^{R}
	r^{N-1} \abs{u(r)}^{p} \, dr,
\end{align*}
and hence we have that for radial functions the norm in $\mathcal{A}$
is, up to a constant, the norm in $H^{1}_{0}(B(0,R))$ and the norms
$\norml[N-1,p]{ \cdot }$ are the $L^{p}(B(0,R))$ norms. The Hilbert
space $\mathcal{A}$ can be identified with the subspace of
$H^{1}_{0}(B(0,R))$ consisting of radial functions.

Let us now recall some well known properties of $\mathcal{A}$, which
plays an essential role in the proof of the main result. One can find 
a more general statement -- and a sketch of the proofs -- in
\cite{Bonheure_Gomes_Sanchez_2005}.

\begin{lem}[Hardy's inequality for radial functions, Proposition
	1 of \cite{Bonheure_Gomes_Sanchez_2005}]
	\label{estimate1}
	Assume that $N \geq 3$. Then for any $u\in \mathcal{A}$ it holds
	\begin{equation*}
		\int_{0}^{R}r^{N-3}u^{2}dr\leq \frac{4}{(N-2)^{2}}
		\normh{u}^{2}.
	\end{equation*}
\end{lem}

\begin{lem}[Corollary 2 of \cite{Bonheure_Gomes_Sanchez_2005} and 
	Lemma 2.2 of \cite{Coelho_Corsato_Rivetti_2014}]
	\label{estimate2}
	Let $N \geq 2$ and $m > (N-2)/2$. Then for all $u\in \mathcal{A}$
	\begin{equation*}
		\norml[\infty]{r^{m}u} \leq C(N,m, R) \normh{u}.
	\end{equation*}
\end{lem}

The following Lemma is the Sobolev embedding of $H^{1}_{0}(B(0,R))$ in
$L^{p}(B(0,R))$ in case of radial functions.
\begin{lem}[Proposition 3 of \cite{Bonheure_Gomes_Sanchez_2005} and 
	Lemma 2.3 of \cite{Coelho_Corsato_Rivetti_2014}]
	\label{estimate3}
	If $N \geq 3$, $p \in [1, \frac{2N}{N-2})$ and $N = 2$, $p \in [1,
	+\infty)$. Then there exists a constant $C(N,p,R)$ such that for
	any $u\in \mathcal{A}$,
	\begin{equation}
		\label{eq:poincare}
		\norml[N-1,p]{u}^{p} \leq C(N,p,R) \normh{u}^{p}.
	\end{equation}
\end{lem}

\begin{remark}
	The inequality \eqref{eq:poincare} is actually the
	\emph{Poincar\'e inequality}. Shows that the norm $\normh{u}$ is
	actually equivalent -- for radial functions -- to the usual norm
	in $H^{1}_{0}(B(0,R))$.
\end{remark}

We recall the following compact embedding, equivalent to the Sobolev
embedding of $H^{1}_{0}$ in $L^{p}$ for the ball of radius $R$.
\begin{lem}
	\label{estimate}
	Assume that $p\in[1, \frac{2N}{N-2})$ if $N\geq 3$ and $p\in[1,
	+\infty)$ if $N=2$. Then the following embedding is compact:
	\begin{equation*}
		\mathcal{A} \hookrightarrow L^{p}_{N-1}.
	\end{equation*}
\end{lem}

We finally introduce the set $K$, in which we will look for 
solutions of our problem.
\begin{equation*}
	K = \settc{u \in \mathcal{A}}{\norml[\infty]{u'} \leq 1}.
\end{equation*}
Remark that $K \subset W^{1,\infty}(0,R) \subset W^{1,2}(0,R)$. 
\begin{lem}
	\label{Kclose}
	$K$ is a closed, convex and bounded subset of $\mathcal{A}$. We 
	also have that 
	\begin{equation*}
		\norml[\infty]{u} \leq R \qquad \text{ for all } u \in K.
	\end{equation*}
\end{lem}

\begin{proof}
	The convexity is immediate. Let us remark that the functions in
	$K$ are Lipschitz with Lipschitz constant $1$, and $K \subset
	W^{1,\infty}(0,R)$. As a consequence we have that $\normh{u}^{2}
	\leq \int_{0}^{R} r^{N-1} \, dr = \frac{R^{N}}{N}$ and also
	$\abs{u(r)} \leq \int_{r}^{R} \abs{u'(r)} \, dr \leq R$. Then the
	family of functions $K$ is equicontinuous and equibounded.
	
	Let $\{u_{n}\} \subset K$, $u_{n} \to u$ in $\mathcal{A}$. We can
	apply Ascoli-Arzel\`a to deduce that a subsequence converges
	uniformly to a function $u$, and also $u$ is 1-Lipschitz, and we
	immediately deduce that $\norml[\infty]{u'} \leq 1$ and $u \in K
	\subset W^{1,\infty} \subset \mathcal{A}$.
\end{proof}

\section{The variational problem}
\label{sec:variational}

In this section we introduce the notion of weak solution of problem
\eqref{RMCE} and the related variational problem, and prove some of
their properties.
 
We will always assume in this section that $f$ satisfies \ref{it:fC},
\ref{it:fat0}, \ref{it:AR} and that $b$ satisfies \ref{it:b}.

\begin{definition}
	\label{weak}
	A weak solution of \eqref{RMCE} is a function $u\in K$ such that
	for all $\varphi \in K$,
	\begin{equation}
		\label{eq:weaksol}
		\int^{R}_{0} \frac{r^{N-1} u' \varphi'}{\sqrt{1 -
		\abs{u'}^{2}}} \, dr = \int^{R}_{0} r^{N-1} \bigl(\lambda
		b(r)\abs{u}^{q-2}u + f(r,u) \bigr) \varphi \, dr.
	\end{equation}
\end{definition}

\begin{remark}
	\eqref{eq:weaksol} holds for all $\varphi \in K$ if and only if 
	it holds for all $\varphi \in W^{1,\infty} \bigcap \mathcal{A}$.
\end{remark}

\begin{lem}
	\label{C1}
	If $u$ is a weak solution, then $u \in C^{1}([0,R])$,
	\begin{align}
		&\int^{R}_{0} \frac{r^{N-1} \abs{u'}^{2}}{\sqrt{1 -
		\abs{u'}^{2}}} \, dr < +\infty \label{eq:stimau2}, \\
		&\int^{R}_{0} \frac{r^{N-1}\abs{u'}}{\sqrt{1-\abs{u'}^{2}}} \,
		dr < +\infty \label{eq:stimau1}
	\end{align}
	and there exists $0 < \varepsilon < 1$ such that $\abs{u'(r)}< 1 -
	\varepsilon$ for all $r\in [0,R]$.
	
	Moreover for a.e. $r \in (0,R)$,
	\begin{equation}
		\label{eq:equationae}
		-\left( \frac{r^{N-1}u'}{\sqrt{1 - \abs{u'}^{2}}}\right)' =
		r^{N-1}\bigl(\lambda b(r)\abs{u}^{q-2}u + f(r,u)\bigr).
	\end{equation}
\end{lem}

\begin{proof}
	The following proof is similar to \cite[Proposition
	3.3]{Bonheure_dAvenia_Pomponio_2016}.
	
	Since $u\in K$, $f$ is continuous and $\abs{u(r)} \leq R$ and we
	have that
	\begin{equation*}
		r^{N-1}\left(\lambda b(r)\abs{u(r)}^{q-2}u + f(r,u(r))\right) \in
		L^{\infty}(0,R)
	\end{equation*}

	Taking $\varphi=u$ in Definition \ref{weak}, follows that
	\begin{equation*}
		\int^{R}_{0} \frac{r^{N-1}\abs{u'}^{2}}{\sqrt{1-\abs{u'}^{2}}}
		\, dr = \int^{R}_{0}r^{N-1}\left(\lambda
		b(r)\abs{u}^{q}+f(r,u)u\right)\, dr < +\infty.
	\end{equation*}
	and
	\begin{align*}
		\int^{R}_{0} \frac{r^{N-1}\abs{u'}}{\sqrt{1-\abs{u'}^{2}}} \,
		dr &= \int_{(0,R) \cap \{\abs{u'} \leq\frac{1}{2}\}}
		\frac{r^{N-1}\abs{u'}}{\sqrt{1 - \abs{u'}^{2}}} \, dr \\
		&\qquad + \int_{(0,R)\cap \{\abs{u'} >\frac{1}{2}\}}
		\frac{r^{N-1}\abs{u'}}{\sqrt{1-\abs{u'}^{2}}}\, dr\\
		&\leq C\left(\int^{R}_{0}r^{N-1}\,
		dr+\int^{R}_{0}\frac{r^{N-1}\abs{u'}^{2}}{\sqrt{1-\abs{u'}^{2}}}\,
		dr\right) < +\infty,
	\end{align*}
	where $C$ is a constant, and \eqref{eq:stimau2} and
	\eqref{eq:stimau1} follow.

	From the fact that $u$ is a weak solution, see \eqref{eq:weaksol},
	we deduce that
	\begin{equation*}
		h_{1}(r) \equiv \frac{r^{N-1}u'}{\sqrt{1-\abs{u'}^{2}}} \in
		W^{1,\infty}(0,R),
	\end{equation*}
	and that for a.e. $r \in (0,R)$,
	\begin{equation}
		\label{eq:equationaeh}
		-h_{1}'(r) = -\left( \frac{r^{N-1}u'}{\sqrt{1 -
		\abs{u'}^{2}}}\right)' = r^{N-1}\bigl(\lambda
		b(r)\abs{u}^{q-2}u + f(r,u)\bigr) \leq Cr^{N-1}
	\end{equation}
	and \eqref{eq:equationae} follows. The continuous embedding of
	$W^{1,\infty}(0,R)$ in $C[0,R]$ implies that
	\begin{equation*}
		h_{1}(r) := \frac{r^{N-1}u'(r)}{\sqrt{1 - \abs{u'(r)}^{2}}} \in
		C[0,R].
	\end{equation*}
	and
	\begin{equation*}
		\frac{u'(r)}{\sqrt{1 - \abs{u'(r)}^{2}}} :=
		\frac{h_{1}(r)}{r^{N-1}} \in C(0,R].
	\end{equation*}
	Now we study the behavior of $u'(r)$ as $r$ tends to $0^{+}$. Fix
	$\varphi \in K$. Integrating by parts \eqref{eq:weaksol} we deduce
	that
	\begin{equation}
		\label{R0}
		-h_{1}(0) \varphi(0) = \int^{R}_{0} \left(h_{1}'(r) + r^{N-1}
		\left(\lambda b(r)\abs{u(r)}^{q-2}u(r) + f(r,u(r)) \right)
		\right)\varphi\, dr.
	\end{equation}
	and from the equation \eqref{eq:equationaeh} we obtain
	\begin{equation*}
		-h_{1}(0) \varphi(0) = 0,
	\end{equation*}
	which implies that $h_{1}(0) = 0$ since $\varphi \in K$ is
	arbitrary.
	
	Integrating \eqref{eq:equationaeh} from $0$ to $r$ with $0 < r \leq
	R$ that
	\begin{equation*}
		\abs{h_{1}(r)} \leq \int^{r}_{0} t^{N-1} \left(\lambda b_{1}
		\abs{u(t)}^{q-1} + \abs{f(t, u(t))} \right) \, dt \leq C
		\frac{r^{N}}{N}.
	\end{equation*}
	and
	\begin{equation}
		\label{uprime}
		\abs{r^{N-1}u'(r)} \leq \left|\frac{r^{N-1}u'(r)}{\sqrt{1 -
		\abs{u'(r)}^{2}}}\right| \leq C\frac{r^{N}}{N}.
	\end{equation}
	So we have
	\begin{equation*}
		\abs{u'(r)} \leq C \frac{r}{N} \qquad \text{for all } r \in
		[0,R].
	\end{equation*}
	Consequently, $u\in C^{1}[0,R]$ with $u'(0)=0$. Since $h_{1}(r) \in
	C([0,R])$ and $u'(0) = 0$ it follows that there exists $0<
	\varepsilon <1$ such that $\abs{u'(r)} < 1 - \varepsilon$ for all
	$r\in [0,R]$.
\end{proof}

\begin{lem}
	\label{lem:weakclassical}
	Any weak solution $u$ is of class $C^{2}[0,R]$ is a classical
	solution of \eqref{RMCE}.
\end{lem}

\begin{proof}
	This proof can be found in \cite[Remark
	3.7]{Coelho_Corsato_Rivetti_2014}. We give it here only for
	reader's convenience. Follows from Lemma \ref{C1} that a weak
	solution $u$ is in $C^{1}[0,R]$ and
	$\frac{r^{N-1}u'}{\sqrt{1-\abs{u'}^{2}}} \in C[0,R]$. From the
	equation \eqref{eq:equationae} follows then that
	$\frac{r^{N-1}u'}{\sqrt{1 - \abs{u'}^{2}}} \in C^{1}[0,R]$.
	Now we want to prove $\frac{u'}{\sqrt{1-\abs{u'}^{2}}}\in
	C^{1}[0,R]$. Let
	\begin{equation*}
		h_{2}(r) = \lambda b(r) \abs{u(r)}^{q-2}u(r) + f(r,u(r)).
	\end{equation*}
	For any $0 < r \leq R$, integrating the equation
	\eqref{eq:equationaeh} from $0$ to $r$, we deduce that
	\begin{equation*}
		\frac{u'(r)}{\sqrt{1 - \abs{u'(r)}^{2}}} =
		-\int_{0}^{r}(\frac{s}{r})^{N-1}h_{2}(s)\, ds \qquad 
		\text{for all } r \in (0,R].
	\end{equation*}
	We have that $\frac{u'}{\sqrt{1-\abs{u'}^{2}}} \in C^{1}((0,R]) 
	\cup C([0,R])$. It
	remains to study the behavior of
	$\left(\frac{u'}{\sqrt{1-\abs{u'}^{2}}}\right)'$ as $r$ tends to
	$0^{+}$. Let's show that the following limit holds
	\begin{equation*}
		\lim_{r \to 0^{+}}\frac{1}{r}\frac{u'(r)}{\sqrt{1 -
		\abs{u'(r)}^{2}}} = -\frac{h_{2}(0)}{N}.
	\end{equation*}
	Fix $\tau>0$. By continuity of $h_{2}(r)$ at $0$, there exists
	$r_{1}>0$ such that $\abs{h_{2}(s)-h_{2}(0)}<\tau$ for all $s\in
	[0,r_{1})$. Taking $r\in(0,r_{1})$, we have
	\begin{align*}
		\labs{\frac{h_{2}(0)}{N} + \frac{1}{r} \frac{u'(r)}{\sqrt{1 -
		\abs{u'(r)}^{2}}}} &= \labs{\frac{h_{2}(0)}{N} -
		\frac{1}{r}\int_{0}^{r}(\frac{s}{r})^{N-1}h(s)\, ds} \\
		&\leq \frac{\tau}{r^{N}} \int_{0}^{r}s^{N-1}\, ds =
		\frac{\tau}{N}.
	\end{align*}
	So $\frac{u'}{\sqrt{1-\abs{u'}^{2}}}\in C^{1}([0,R])$ and $u'\in
	C^{1}([0,R])$. Consequently, $u\in C^{2}([0,R])$ is a classical
	solution of \eqref{RMCE}.
\end{proof}

We now give the proof of another basic property of the the nontrivial
solutions of \eqref{RMCE}.
\begin{lem}
	\label{lem:positivity} 
	
	If $u$ is a positive weak solution, then $u'(r) < 0$ for any $r
	\in (0,R]$. If $u$ is a negative solution, then $u'(r)>0$ for any
	$r \in (0,R]$.
\end{lem}

\begin{proof}
	Let us remark that \ref{it:AR} implies that $f(r,s)s \geq 0$ for
	all $(r,s) \in [0,  R] \times [0,R] $. Let $u$ be a
	positive weak solution. We know from Lemma \ref{lem:weakclassical}
	that $u\in C^{2}[0,R]$ is a classical solution of \eqref{RMCE},
	hence for every $r \in [0,R]$, we have
	\begin{equation*}
		-\left(\frac{r^{N-1}u'}{\sqrt{1-\abs{u'}^{2}}}\right)' = r^{N-1}
		\lambda b(r)\abs{u}^{q-2}u + r^{N-1}f(r,u).
	\end{equation*}
	Since for all $r \in (0,R)$, $f(r,u(r)) \geq 0$ we deduce that
	\begin{equation*}
		-\left(\frac{r^{N-1}u'(r)}{\sqrt{1 - \abs{u'(r)}^{2}}}\right)' > 0
	\end{equation*}
	and hence we have, for all $0<r \leq R$, that
	\begin{equation*}
		\frac{r^{N-1}u'(r)}{\sqrt{1 - \abs{u'(r)}^{2}}} <
		\frac{r^{N-1}u'(r)}{\sqrt{1 - \abs{u'(r)}^{2}}}\Big|_{r=0} = 0
	\end{equation*}
	which means $u'(r)<0$ for any $r \in (0,R]$. The proof is similar
	if $u$ is a negative solution.
\end{proof}

We will now start building the functional whose critical points will
correspond to weak solutions of \eqref{RMCE}. 

First of all we observe that, since
\begin{equation}
	\label{inequality}
	\frac{1}{2} t \leq 1-\sqrt{1-t} \leq t, \text{ for all }
	t\in[0,1],
\end{equation}
we have that
\begin{equation*}
	\frac{1}{2} \int^{R}_{0} r^{N-1} \abs{u'}^{2} \, dr \leq
	\int^{R}_{0} r^{N-1} (1 - \sqrt{1 - \abs{u'}^{2}}) \, dr \leq
	\int^{R}_{0} r^{N-1} \abs{u'}^{2} \, dr
\end{equation*}
and we can define $\Psi \colon \mathcal{A} \to (-\infty,+\infty]$,
\begin{equation*}
	\Psi(u) = 
	\begin{cases}
		\int^{R}_{0} r^{N-1} (1 - \sqrt{1 - \abs{u'}^{2}}) \, dr
		&\text{if } u \in K, \\
		+\infty &\text{otherwise}
	\end{cases}.
\end{equation*}

\begin{lem}
	\label{wlsc}
	$\Psi$ is weakly lower semi-continuous on $\mathcal{A}$.
\end{lem}

\begin{proof}
	In this proof we follow ideas from Bereanu, Jebelean, and Mawhin
	\cite{Bereanu_Jebelean_Mawhin_2014}.
	
	It is enough to show that $\Psi(u) \leq \liminf_{n\to+\infty}
	\Psi(u_{n})$ for sequences $\{u_{n}\} \subset K$ which weakly
	converge to some $u \in K$. By Ascoli-Arzel\`{a} Theorem we know
	that a subsequence $u_{n} \to u$ uniformly in $C[0,R]$. Since
	$\norml[\infty]{u'_{n}} \leq 1$, we know that a subsequence (which
	we still denote $u'_{n}$) converges in the weak*
	$\sigma(L^{\infty}, L^{1})$ topology to a function $v \in
	L^{\infty}$, that is
	\begin{equation*}
		\int_{0}^{R} u'_{n}(r) w(r) \, dr \to \int_{0}^{R} v(r) w(r)
		\, dr, \qquad \text{for all } w \in L^{1}(0,R).
	\end{equation*}
	
	Take any $\varphi \in C^{\infty}_{0}(0,R)$. We have that
	\begin{multline*}
		\int^{R}_{0} u(r) \varphi'(r) \, dr = \lim_{n \to +\infty}
		\int^{R}_{0} u_{n}(r) \varphi'(r) \, dr \\
		= -\lim_{n \to +\infty} \int^{R}_{0} u'_{n}(r) \varphi(r) \,
		dr = -\int^{R}_{0} v(r) \varphi(r) \, dr,
	\end{multline*}
	and hence $v(r) = u'(r)$.

	Since the function $t \mapsto 1 -\sqrt{1-t^{2}}$ is convex, we
	deduce that for all $t \in (0,1)$,
	\begin{multline}
		\label{convexity1}
		\int^{R}_{0} r^{N-1} (1 - \sqrt{1 - \abs{u_{n}'}^{2}} ) \, dr
		\\
		\geq \int^{R}_{0} r^{N-1} (1 - \sqrt{1 - \abs{t u'}^{2}}) \,
		dr + \int^{R}_{0} r^{N-1} \frac{t u'( u'_{n} - t u')}{\sqrt{ 1
		- \abs{t u'}^{2}}} \, dr.
	\end{multline}
	Since $\abs{tu'(r)} \leq t < 1$ for all $r \in [0,R]$ if $t \in
	(0,1)$ we have that
	\begin{equation*}
		 w_{t}(r) \equiv r^{N-1} \labs{\frac{t u'(r)}{\sqrt{ 1 - \abs{t
		 u'}^{2}(r)}}} \leq \frac{R^{N-1}t}{\sqrt{1-t^{2}}}
	\end{equation*}
	and $w_{t} \in L^{\infty} \subset L^{1}$ follows. Then
	\begin{equation*}
		\int^{R}_{0} r^{N-1} \frac{t u'(r)}{\sqrt{ 1 - \abs{t
		u'(r)}^{2}}} u'_{n}(r) \, dr \to \int^{R}_{0} r^{N-1} \frac{t
		\abs{u'(r)}^{2}}{\sqrt{ 1 - \abs{t u'(r)}^{2}}} \, dr
	\end{equation*}
	and from \eqref{convexity1} we deduce that for all $t \in (0,1)$
	\begin{multline*}
		\liminf_{n \to +\infty} \int^{R}_{0} r^{N-1} (1 - \sqrt{1 -
		\abs{u_{n}'}^{2}} ) \, dr \\
		\geq \int^{R}_{0} r^{N-1} (1 - \sqrt{1 - \abs{t u'}^{2}}) \,
		dr + \frac{1-t}{t} \int^{R}_{0} r^{N-1}
		\frac{\abs{tu'}^{2}}{\sqrt{ 1 - \abs{t u'}^{2}}} \, dr \\
		\geq \int^{R}_{0} r^{N-1} (1 - \sqrt{1 - \abs{t u'}^{2}}) \,
		dr.
	\end{multline*}
	Since $0 \leq (1 - \sqrt{1 - \abs{t u'}^{2}}) \leq 1$ for all $t
	\in (0,1)$, we can use Lebesgue dominated convergence to pass to
	the limit as $t \to 1$, showing that
	\begin{equation*}
		\liminf_{n \to +\infty} \int^{R}_{0} r^{N-1} (1 - \sqrt{1 -
		\abs{u_{n}'}^{2}} ) \, dr \geq \int^{R}_{0} r^{N-1} (1 -
		\sqrt{1- \abs{u'}^{2}}) \, dr,
	\end{equation*}
	and $\Psi$ is weakly lower semicontinuous.
\end{proof}

Since we are interested in positive and negative solutions, and since
functions in $K$ are bounded by $R$ we introduce the following
modifications of $f$:
\begin{equation*}
	\fp(r,s) =
	\begin{cases}
	   f(r,s) &\text{if } 0 \leq s \leq R \\
	   -f(r, R)(s-R-1) &\text{if } R < s < R+1 \\
	   0 &\text{if } s < 0 \text{ or } s \geq R+1
	\end{cases}
\end{equation*}
and
\begin{equation*}
	\fm(r,s) =
	\begin{cases}
	   f(r,s) &\text{if } -R \leq s \leq 0 \\
	   f(r, -R)(s+R+1) &\text{if } -R-1 < s < -R \\
	   0 &\text{if } s > 0 \text{ or } s \leq -R-1
	\end{cases}
\end{equation*}
If $f$ satisfies \ref{it:fC} and \ref{it:fat0}, the same is true for
$\fpm$, and we also have that $\fpm(r,s) s \geq 0$ for all $(r,s) \in
[0,R] \times \mathbb{R}$ so that all the results we have proved in
this section hold for $\fpm$. Since $\fpm(r, s)$ is continuous there
exists a positive constant $c = c(R)$, such that
\begin{equation}
	\label{fbounded}
	\abs{\fpm(r,s)} \leq c(R) \qquad \text{for all } r \in [0,R] \text{
	and } s \in \mathbb{R}.
\end{equation}

Corresponding to the modified nonlinear terms $\fpm$ we consider the
following modified problems
\begin{equation}
	\label{eRMCE+}
	\tag{$BP^{+}$}
	\begin{cases}
		- \left(\frac{r^{N-1} u'}{\sqrt{1 - \abs{u'}^{2}}} \right)' =
		\lambda r^{N-1} \abs{u}^{q-2} u^{+} + r^{N-1} \fp(r,u) & r \in
		(0,R)\\
		u'(0)=u(R)=0,
	\end{cases}
\end{equation}
where $u^{+}=\max \{0,u\}$ and 
\begin{equation}
	\label{eRMCE-}
	\tag{$BP^{-}$}
	\begin{cases}
		- \left(\frac{r^{N-1} u'}{\sqrt{1 - \abs{u'}^{2}}} \right)' =
		- \lambda r^{N-1} \abs{u}^{q-2} u^{-} + r^{N-1} \fm(r,u) & r \in
		(0,R)\\
		u'(0)=u(R)=0,
	\end{cases}
\end{equation}
where $u^{-}= \max\{0, - u\} = (-u)^{+}$.
\begin{lem}
	\label{positive}
	Assume that $u$ is a nontrivial weak solution of equation
	\eqref{eRMCE+}. Then $u$ is a nonnegative weak (and hence
	classical) solution of \eqref{RMCE}.
	
	Similarly, a nontrivial weak solution of equation \eqref{eRMCE-}
	is actually a nonpositive weak (and hence classical) solution of
	\eqref{RMCE}.
\end{lem}

\begin{proof}
	Let us prove the first assertion. Taking $\varphi = u^{-}$ in the
	definition of weak solution \eqref{eq:weaksol} for \eqref{eRMCE+}
	by $u^{-}$ we get
	\begin{equation*}
		-\int^{R}_{0} \frac{r^{N-1}\abs{(u^{-})'}^{2}}{\sqrt{1 -
		\abs{u'}^{2}}} \, dr = \int^{R}_{0}r^{N-1} \left( \lambda
		\abs{u}^{q-2}u^{+} + \fp(r,u) \right) u^{-}\, dr = 0,
	\end{equation*}
	which implies $u^{-} = 0$ and thus $u = u^{+}$ is a non-negative
	weak solution of \eqref{eRMCE+}. Moreover since $u \in K$ we have
	that $0 \leq u(r) \leq R$. It follows that $u^{+}(r) = u(r)$ and
	$\fp(r,u(r)) = f(r,u(r))$ for all $r \in [0, R]$ and the Lemma
	follows.
\end{proof}

We let $\Fpm \colon \mathcal{A} \to \mathbb{R}$,
\begin{equation*}
	\Fpm(u) = - \frac{\lambda}{q} \int^{R}_{0} r^{N-1} b(r)
	\abs{u^{\pm}}^{q} \, dr - \int^{R}_{0} r^{N-1} \ffpm(r, u)\, dr
\end{equation*}
where
\begin{equation*}
	\ffpm(r,s) = \int_{0}^{s} \fpm(r,t) \, dt.
\end{equation*}
Since $\fpm$ are bounded and continuous, it is a standard fact that
$\Fpm \colon \mathcal{A} \to \mathbb{R}$ are $C^{1}$
maps and that for all $u$, $v \in \mathcal{A}$ we have
\begin{equation*}
	\scalarh{(\Fpm)'(u)}{v} = - \lambda \int^{R}_{0} r^{N-1} b(r)
	\abs{u^{\pm}}^{q-2} u^{\pm} v \, dr - \int^{R}_{0} r^{N-1}
	\fpm(r,u) v \, dr.
\end{equation*}

The following property is an immediate consequence of the compactness
of the embedding of $\mathcal{A}$ in $L^{p}_{N-1}$:
\begin{lem}
	\label{Fcompact}
	The operator $(\Fpm)'$ is completely continuous
	on $\mathcal{A}$, i.e. if $u_{n} \rightharpoonup u$, then
	$(\Fpm)'(u_{n}) \to (\Fpm)'(u).$
\end{lem}

Then the Euler-Lagrange functional $\Ipm \colon \mathcal{A} \to
(-\infty, +\infty]$ can be defined as
\begin{equation*}
	I^{\pm}_{\lambda}(u) = \Psi(u) + \Fpm(u).
\end{equation*}

Now we want to explore properties of functionals $I^{\pm}_{\lambda}$.
The proof of the following two Lemmas are similar to those in
\cite[Proposition 2.7 and proof of Theorem
1.4]{Bonheure_dAvenia_Pomponio_2016}.

\begin{lem}
	\label{lem:criticalinL1}
	Assume $u\in \mathcal{A}$ is a critical point of
	$I^{\pm}_{\lambda}$. Then $u$ satisfies
	\begin{equation}
		\label{eq:stimau2cr}
		\int^{R}_{0} \frac{r^{N-1} \abs{u'}^{2}}{\sqrt{1 -
		\abs{u'}^{2}}} \, dr < +\infty
	\end{equation}
	and 
	\begin{equation}
		\label{eq:stimau1cr}
		\int^{R}_{0} \frac{r^{N-1}\abs{u'}}{\sqrt{1-\abs{u'}^{2}}} \,
		dr < +\infty.
	\end{equation}
	In particular we have that
	\begin{equation*}
		E = \settc{r \in (0,R)}{\abs{u'(r)} = 1}
	\end{equation*}
	is a null set with respect to Lebesgue measure.
\end{lem}
\begin{proof}
	Let us assume $u$ is a critical point of $I^{+}_{\lambda}$. The
	other case is similar. Since $u$ is a critical point of
	$I^{+}_{\lambda}$, by Definition \ref{def:d1}, we have that $u \in
	K$ and for any $v \in K$
	\begin{multline*}
		\int^{R}_{0} r^{N-1}(-\sqrt{1 - \abs{v'}^{2}} + \sqrt{1 -
		\abs{u'}^{2}}) \, dr \\
		\geq \int^{R}_{0} r^{N-1} \left( \lambda b(r)\abs{u^{+}}^{q-2}
		u^{+} + \fp(r,u) \right)(v-u) \, dr.
	\end{multline*}
	Let $v_{t} = (1-t)u$. Then $v_{t} \in K$ for all $t\in(0,1)$ and then
	\begin{multline*}
		\int^{R}_{0} r^{N-1} \frac{\left(\sqrt{1 - (1 - t)^{2}
		\abs{u'}^{2}} - \sqrt{1 - \abs{u'}^{2}} \right)}{t} \\
		\leq \int^{R}_{0} r^{N-1} \left(\lambda b(r)\abs{u^{+}}^{q} +
		\fp(r,u)u \right) \, dr.
	\end{multline*}
	We deduce from this that
	\begin{multline*}
		\int^{R}_{0} \frac{r^{N-1} (\abs{u'}^{2} - \abs{(1 -
		t)u'}^{2})}{t\left(\sqrt{1-\abs{(1-t)u'}^{2}} + \sqrt{1 -
		\abs{u'}^{2}}\right)} \\
		\leq \int^{R}_{0}r^{N-1} \left( \lambda b(r)\abs{u^{+}}^{q} +
		\fp(r,u) u \right)\, dr.
	\end{multline*}

	Since $u \in K \subset L^{\infty}$ and $f$ is a continuous
	function we have that
	\begin{equation*}
		\int^{R}_{0}r^{N-1} \left( \lambda b(r)\abs{u^{+}}^{q} +
		\fp(r,u) u \right)\, dr < +\infty
	\end{equation*} 
	and by Fatou's Lemma, letting $t \to 0^{+}$, we obtain
	\eqref{eq:stimau2cr}.

	Proceeding now as in the proof of \eqref{eq:stimau1} in Lemma
	\ref{C1} we deduce from \eqref{eq:stimau2cr} that
	\eqref{eq:stimau1cr} holds.
\end{proof}

\begin{lem}
	\label{lem:criticalAREsolutions}
	If $u$ is a critical point of $I^{+}_{\lambda}$, then $u$ is a
	weak solution of \eqref{eRMCE+}, and hence a nonnegative classical
	solution of \eqref{RMCE}.
	
	If $u$ is a critical point of $I^{-}_{\lambda}$, then $u$ is a
	weak solution of \eqref{eRMCE-}, and hence a nonpositive classical
	solution of \eqref{RMCE}.

\end{lem}

\begin{proof}
	For $k \in \mathbb{N}\backslash \{0\}$, we define
	\begin{equation}\label{Ek}
		E_{k} = \settc{r\in(0,R)}{\abs{u'(r)} \geq 1-\frac{1}{k}}.
	\end{equation}
	According to Lemma \ref{lem:criticalinL1},
	\begin{equation}\label{E}
		E = \settc{r\in(0,R)}{\abs{u'(r)}= 1}
	\end{equation}
	is a null set with respect to Lebesgue measure. Since $E_{k} 
	\supset E_{k+1}$ and $E = \bigcap_{k} E_{k}$ we have that
	$\lim_{k \to +\infty} \abs{E_{k}} = \abs{E} = 0$.

	Take $\varphi\in K$ and let
	\begin{equation*}
		\varphi_{k}(r)=-\int^{R}_{r}\varphi'(s)[1-\chi_{E_{k}}(s)] \, ds,
	\end{equation*}
	where
	\begin{equation*}
		\chi_{E_{k}}(r)=
		\begin{cases}
			1 &\text{if } r\in E_{k}, \\
			0 &\text{if } r \notin E_{k}\\
		\end{cases}.
	\end{equation*}
	For all $\abs{t} \leq \frac{1}{k}$ we have $u + t\varphi_{k}\in
	K$. Indeed, since $(u + t\varphi_{k})'(r) = u'(r) +
	t\varphi'(r)[1-\chi_{E_{k}}(r)]$, we have, if $r \in E_{k}$,
	\begin{equation*}
		\abs{(u+t\varphi_{k})'(r)} = \abs{u'(r)}\leq 1
	\end{equation*}
	while if $r\notin E_{k}$, letting $\abs{t} \leq \frac{1}{k}$, we
	have
	\begin{equation*}
		\abs{(u+t\varphi_{k})'(r)} \leq \abs{u'(r)} + \abs{t}
		\abs{\varphi'(r)} < 1 - \frac{1}{k} + \frac{1}{k} = 1.
	\end{equation*}
	
	We have that for all $0 < \abs{t} \leq \frac{1}{k}$ 
	\begin{align*}
		\frac{\Psi(u+t\varphi_{k}) - \Psi(u)}{t} &= \int^{R}_{0}
		r^{N-1} \frac{\sqrt{1-\abs{u'}^{2}} - \sqrt{1-\abs{(u +
		t\varphi_{k})'}^{2}}}{t}\, dr \\
		&= \int^{R}_{0} r^{N-1} \frac{2u'(r)\varphi'_{k}(r) +
		t\varphi'_{k}(r)^{2}}{\sqrt{1-\abs{u'}^{2}} + \sqrt{1-\abs{(u +
		t\varphi_{k})'}^{2}}}\, dr \\
		&= \int_{(0,R) \setminus E_{k}} r^{N-1} \frac{2u'(r)\varphi'_{k}(r) +
		t\varphi'_{k}(r)^{2}}{\sqrt{1-\abs{u'}^{2}} + \sqrt{1-\abs{(u +
		t\varphi_{k})'}^{2}}}\, dr
	\end{align*}
	Since $\abs{u'(r)} < 1 - \frac{1}{k}$ in $(0,R) \setminus E_{k}$ 
	we can pass to the limit as $t \to 0$ to deduce that for all $k 
	\in \mathbb{N}$
	\begin{equation*}
		\lim_{t \to 0} \frac{\Psi(u+t\varphi_{k}) - \Psi(u)}{t} =
		\int_{(0,R)} r^{N-1}
		\frac{u'(r)\varphi'_{k}(r)}{\sqrt{1-\abs{u'}^{2}}}\, dr.
	\end{equation*}
	From the fact that $u$ is a critical point we deduce that
	\begin{equation*}
		\scalarh{(\Fpm)'(u)}{t\varphi_{k}} + \Psi(u+t\varphi_{k})
		- \Psi(u) \geq 0
	\end{equation*}
	and passing to the limit as $t \to 0\pm$ we obtain that for all 
	$k \in \mathbb{N}$
	\begin{align*}
		0 = \lim_{t \to 0}
		&\left(\scalarh{(\Fpm)'(u)}{\varphi_{k}} +
		\frac{\Psi(u+t\varphi_{k}) - \Psi(u)}{t} \right)\\
		&= \int^{R}_{0}r^{N-1} \left(\frac{u'\varphi_{k}'}{\sqrt{1 -
		\abs{u'}^{2}}} -\lambda b(r) \abs{u^{\pm}}^{q-2} u^{\pm}
		\varphi_{k} - \fpm(r, u)\varphi_{k}\right) \, dr
	\end{align*}
	Since $E_{k+1} \subset E_{k}$ and $\abs{E_{k}}\to0$ as $k \to
	+\infty$, then $\chi_{E_{k}}\to 0$ a.e. in $(0,R)$. Since
	\eqref{eq:stimau1cr} holds we deduce from Lebesgue's Dominated
	Convergence Theorem that
	\begin{equation*}
		\int^{R}_{0} r^{N-1} \frac{u'\varphi'}{\sqrt{1 -
		\abs{u'}^{2}}}[1 - \chi_{E_{k}}]\, dr \to \int^{R}_{0}r^{N-1}
		\frac{u'\varphi'}{\sqrt{1-\abs{u'}^{2}}}\, dr.
	\end{equation*}

	Since also
	\begin{multline*}
		\int^{R}_{0}r^{N-1} \left(\lambda b(r)\abs{u^{\pm}}^{q-2}
		u^{\pm} + \fpm(r,u) \right) \varphi_{k} \, dr \\
		\to \int^{R}_{0}r^{N-1}
		\left(\lambda b(r) \abs{u^{\pm}}^{q-2} u^{\pm} + \fpm(r,u)
		\right) \varphi \, dr,
	\end{multline*}
	we deduce that any $\varphi \in K$ holds that
	\begin{equation}
		\label{weaks}
		\int^{R}_{0} \frac{r^{N-1}u' \varphi'}{\sqrt{1 -
		\abs{u'}^{2}}} \, dr = \int^{R}_{0}r^{N-1} \left(\lambda
		b(r)\abs{u^{\pm}}^{q-2}u^{\pm} + \fpm(r,u) \right) \varphi \,
		dr.
	\end{equation}
	i.e.~$u$ is a weak solution of \eqref{eRMCE+} or \eqref{eRMCE-}.
\end{proof}

We finish this section proving that the functional satisfies the 
Palais-Smale compactness condition.
\begin{lem}
	\label{lem:PS}
	Every Palais-Smale sequence $\{u_n\}$ of $I^{\pm}_{\lambda}$ has a
	convergent subsequence in $\mathcal{A}$.
\end{lem}

\begin{proof}
	To fix ideas, let $\{u_n\}$ be a (PS) sequence for $\Ip$. Since
	the Palais-Smale sequence $\{u_n\} \subset K$, then it is bounded
	and we may assume that $u_{n} \rightharpoonup u \in K$. According
	to Lemma \ref{wlsc}, $\Psi$ is weakly lower semi-continuous on
	$\mathcal{A}$, so we have
	\begin{equation}
		\label{wlsc1}
		\Psi(u) \leq \liminf_{n \to +\infty}\Psi(u_{n}).
	\end{equation}
	Let us show that $u$ is a weak solution of \eqref{eRMCE+}.
	For any $v\in \mathcal{A}$, we know that
	\begin{equation*}
		\scalarh{(\Fp)'(u_{n})}{v-u_{n}} + \Psi(v) - \Psi(u_{n}) \geq
		-\varepsilon_{n} \normh{v-u_{n}}.
	\end{equation*}
	Taking the $\limsup$ we obtain
	\begin{equation*}
		\limsup_{n \to +\infty} \left(\scalarh{(\Fp)'(u_{n})}{v -
		u_{n}} + \Psi(v) - \Psi(u_{n})\right) \geq \limsup_{n \to
		+\infty} \left( -\varepsilon_{n} \normh{v-u_{n}} \right),
	\end{equation*}
	and hence
	\begin{equation*}
		\lim_{n \to +\infty} \scalarh{(\Fp)'(u_{n})}{v - u_{n}} +
		\Psi(v) - \liminf_{n \to +\infty} \Psi(u_{n}) \geq -\lim_{n
		\to +\infty} \varepsilon_{n} \normh{v - u_{n}}.
	\end{equation*}
	Since $(\Fp)'(u)$ is complete continuous (see Lemma
	\ref{Fcompact}) and \eqref{wlsc1} holds we find that
	\begin{equation*}
		\scalarh{(\Fp)'(u)}{v - u} + \Psi(v) - \Psi(u)
		\geq 0,
	\end{equation*}
	which means that $u$ is a critical point of $\Ip$. From Lemma
	\ref{lem:criticalAREsolutions}, $u$ is a weak solution of
	\eqref{RMCE} and follows from Lemma \ref{C1} that there exists a
	constant $0 < \varepsilon < 1$ such that $\abs{u'(r)} <
	1-\varepsilon$ for all $r \in [0,R]$.
	
	Now let $v=u$ in \eqref{eq:ps}. We have
	\begin{multline*}
		\scalarh{\Fp(u_{n})}{u - u_{n}} +
		\int^{R}_{0} r^{N-1} \left(\sqrt{1-\abs{u_{n}'}^{2}}
		-\sqrt{1-\abs{u'}^{2}})\right) \, dr \\
		\geq -\varepsilon_{n} \normh{u-u_{n}}.
	\end{multline*}
	Since $u_{n}\rightharpoonup u \in K$ using Lemma \ref{Fcompact}
	again we have that $\scalarh{(\Fp)'(u_{n})}{u - u_{n}} \to 0$ as
	$n \to +\infty$ and we deduce that
	\begin{multline*}
		\lim_{n \to +\infty} \int^{R}_{0} r^{N-1} \frac{\abs{u'}^{2} -
		\abs{u_{n}'}^{2}}{\sqrt{2 \varepsilon - \varepsilon^{2}}} \,
		dr \\
		\geq \lim_{n \to +\infty}\int^{R}_{0} r^{N-1}
		\frac{\abs{u'}^{2} - \abs{u_{n}'}^{2}}{\sqrt{1 - \abs{u'}^{2}}
		+ \sqrt{1 - \abs{u_{n}'}^{2}}} \, dr \geq 0.
	\end{multline*}
	From this we deduce that $\lim_{n \to \infty}\normh{u_{n}} \leq
	\normh{u}$. Since we also have that $\normh{u} \leq \lim_{n \to
	+\infty} \normh{u_{n}}$ it follows that $u_{n} \to u$ in
	$\mathcal{A}$.
\end{proof}

\section{Existence of critical points}

In this section we will prove the existence of critical points for the
functionals $I^{\pm}_{\lambda}$. We assume in this section that
\ref{it:fC}, \ref{it:fat0}, \ref{it:AR}, \ref{it:fsuper} and
\ref{it:b} hold for $f$ (and hence also for $\fpm$). We start by
proving the existence of two minima, one positive and one negative
(the global minima of, respectively, $\Ip$ and $\Imm$).

\begin{lem}
	\label{lem:globalmin}
	For all $\lambda \in (0,\infty)$ and $R > 0$ there exist $u_{+}
	\geq 0$ and $u_{-} \leq 0$ such that
	\begin{equation*}
		\Ipm(u_{\pm}) = \mu_{\pm} = \inf_{u \in \mathcal{A}} \Ipm(u).
	\end{equation*}
	
	Moreover there  exist $\bar{u}_{\pm} \in K$ with
	\begin{equation}
		\label{eq:MPe}
		\normh{\bar{u}_{\pm}}^{2} = \frac{R^{N}}{N} \qquad \text{and}
		\qquad \Ipm(u_{\pm}) = \mu_{\pm} \leq \Ipm(\bar{u}_{\pm}) \leq
		-R^{N}.
	\end{equation}
\end{lem}

\begin{proof}
	From \eqref{fbounded} we know that $\abs{\fpm(r,u)} \leq c$, and
	thus
	\begin{equation*}
		F^{\pm}(r,u) \leq c \abs{u^{\pm}}, \qquad \forall 
		r\in[0, R], u \in \mathbb{R}.
	\end{equation*}
	We have that $\Ipm(u) = +\infty$ if $u \in \mathcal{A} \setminus
	K$. Recalling that $\abs{u(r)} \leq R$ if $u \in K$ we have
	\begin{align*}
		I^{\pm}_{\lambda}(u) &\geq - \frac{\lambda b_{1}}{q}
		\int^{R}_{0} r^{N-1} \abs{u^{\pm}}^{q} \, dr - \int^{R}_{0}
		r^{N-1} c \abs{u^{\pm}} \,dr \\
		&\geq -\frac{\lambda b_{1}}{q} \int^{R}_{0} r^{N-1} R^{q}\, dr
		- \int^{R}_{0} r^{N-1} c R \, dr = -\frac{\lambda b_{1}
		R^{N+q}}{qN}-\frac{c R^{N+1}}{N},
	\end{align*}
	and $I^{\pm}_{\lambda}$ is bounded from below on $\mathcal{A}$.
	Since $\Ipm$ satisfies the Palais-Smale condition, it follows from
	Proposition \ref{minimumcritical} that there exist $u_{\pm} \in K$
	such that
	\begin{equation*}
		I^{\pm}_{\lambda}(u_{\pm}) = \inf_{u \in \mathcal{A}} \Ipm(u).
	\end{equation*}
	
	We now estimate $\mu_{\pm}$.
	Since
	\begin{equation*}
		F(r,u) \geq a_{1}\abs{u}^{\theta} - a_{2}, \qquad \forall (r,u)
		\in [0,R] \times [ -R ,R].
	\end{equation*}
	Let 
	\begin{equation*}
		\bar{u}_{\pm}(r) = \pm R(1 - \frac{r}{R}) \in K.
	\end{equation*}
	We have that 
	\begin{equation*}
		\normh{\bar{u}_{\pm}}^{2} = \int_{0}^{R} r^{N-1}
		\abs{\bar{u}_{\pm}'(r)}^{2} \, dr = \frac{R^{N}}{N}
	\end{equation*}
	and we deduce, using \eqref{eq:stimaaa} of assumption
	\ref{it:fsuper}
	\begin{align*} 
		I_{\lambda}(\bar{u}_{\pm}) &\leq (1 + a_{2}) R^{N}
		\int^{1}_{0} t^{N-1} \, dt - a_{1} R^{N+\theta} \int^{1}_{0}
		t^{N-1} (1-t)^{\theta} \, dt \\
		&\leq (1 + a_{2}) \frac{R^{N}}{N} - a_{1} R^{N + \theta}
		\int^{1}_{0} t^{N-1} (1 - t)^{\theta} \, dt \\
		&= (1 + a_{2}) \frac{R^{N}}{N} - a_{1} R^{N + \theta}
		\frac{\Gamma(N)\Gamma(\theta+1)}{\Gamma(N+\theta+1)} \leq 
		-R^{N}.
	\end{align*}
\end{proof}

We will now show that we can apply the Mountain Pass Theorem to our 
functionals $\Ipm$. 

\begin{lem}
	\label{lem:locmin}
	There exist $\lambda_{*}(R) > 0$ and $\rho_{+}(R) > 0$ and such that
	for all $\lambda \in (0,\lambda_{*}(R))$
	\begin{equation*}
		\Ip(u) \geq \frac{\rho_{+}(R)^{2}}{8} \qquad \text{for all } 
		\normh{u} = \rho_{+}(R).
	\end{equation*}
	Moreover for all $\lambda \in (0,\lambda_{*}(R))$ there is 
	$\rho_{-}(\lambda,R) \in (0, \rho_{+}(R))$ such that 
	\begin{equation*}
			\Ip(u) \geq \frac{\rho_{-}(\lambda,R)^{2}}{8} \qquad
			\text{for all } \rho_{-}(\lambda,R) \leq \normh{u} \leq
			\rho_{+}(R)
	\end{equation*}
	and $\rho_{-}(\lambda,R) \to 0$ as $\lambda \to 0$.
	
	Similar results holds for $\Imm$.
\end{lem}

\begin{proof}
	Take any $\alpha \in (2,\frac{2N}{N-2})$. From \ref{it:fat0}
	follows that all $\epsilon > 0$ there is $c_{\epsilon}$ such that
	\begin{equation*}
		\abs{\ffp(r,u)} \leq \frac{1}{2} \epsilon u^{2} + c_{\epsilon}
		\abs{u}^{\alpha}, \qquad \forall (r,u) \in [0,R] \times [0,R].
	\end{equation*}
	Recalling the inequality \eqref{inequality} we find that
	\begin{align*}
		\Ip(u) &\geq \frac{1}{2}\int^{R}_{0}r^{N-1}\abs{u'}^{2}\, dr -
		\int^{R}_{0} r^{N-1} \left(\frac{\lambda b(r)}{q} \abs{u}^{q}
		+ \frac{1}{2} \epsilon u^{2} +
		c_{\epsilon}\abs{u}^{\alpha}\right) \, dr \\
		&\geq \frac{1}{2}(1 - \epsilon {d_{2}}) \normh{u}^{2} -
		\lambda {d_{q}} \normh{u}^{q} - c_{\epsilon}
		{d_{\alpha}}\normh{u}^{\alpha} \\
		&=\normh{u}^{2}\left(\frac{1}{2}(1 - \epsilon {d_{2}}) -
		\lambda {d_{q}} \normh{u}^{q-2} - c_{\epsilon} {d_{\alpha}}
		\normh{u}^{\alpha-2} \right)
	\end{align*}
	where $c_{\epsilon}$, ${d}_{2} = C(N,2,R)$, ${d}_{q} =
	\frac{b_{1}}{q}C(N,q,R)$ and ${d}_{\alpha} = C(N,\alpha,R)$ depend
	on $R$ (see Lemma \ref{estimate3}). Take $\epsilon$ small enough
	so that $A = \frac{1}{2}(1 - \epsilon {d_{2}}) \geq \frac{3}{8}$.
	Let
	\begin{equation*}
		H_{\lambda,R}(t)= \lambda {d_{q}} t^{q-2} + c_{\epsilon}
		{d_{\alpha}} t^{\alpha-2}
	\end{equation*}
	This function has a minimum at the point 
	\begin{equation*}
		t_{\lambda,R} = \left(\frac{\lambda(2 - q) {d}_{q}}{(\alpha -
		2)c_{\epsilon} {d}_{\alpha}}\right)^{\frac{1}{\alpha-q}}
	\end{equation*}
	whose value is
	\begin{equation*}
		H_{\lambda,R}(t_{\lambda,R}) = 
		(\lambda d_{q})^{\frac{\alpha-2}{\alpha-q}}
		(c_{\epsilon} d_{\alpha})^{\frac{2-q}{\alpha-q}}
		\left(\frac{2-q}{\alpha-2}\right)^{\frac{\alpha-2}{\alpha-q}}
		\left( \frac{\alpha-q}{2-q}\right).
	\end{equation*}
	Let 
	\begin{equation*}
		\beta = \frac{ (\alpha -
		q)^{\frac{\alpha-q}{\alpha-2}}}{(2-q)^{\frac{2 - q}{\alpha -
		2}}(\alpha-2)}
	\end{equation*}
	and take 
	\begin{equation*}
		\lambda_{*}(R) =
		\left(\frac{1}{4}\right)^{\frac{\alpha-q}{\alpha-2}}
		\frac{1}{\beta d_{q}(c_{\epsilon}
		d_{\alpha})^{\frac{2-q}{\alpha-2}}}
	\end{equation*}
	so that
	\begin{equation*}
		t_{\lambda_{*}(R),R} = \left( \frac{2-q}{4c_{\epsilon}(\alpha
		- q) d_{\alpha}} \right)^{\frac{1}{\alpha-2}}.
	\end{equation*}
	If $\lambda < \lambda_{*}(R)$ we have that 
	\begin{align*}
		\rho_{-}(\lambda,R) &\equiv t_{\lambda,R} =
		\left(\frac{\lambda(2 - q) {d}_{q}}{(\alpha - 2)c_{\epsilon}
		{d}_{\alpha}}\right)^{\frac{1}{\alpha-q}} 
		< \left(\frac{\lambda_{*}(R)(2 - q) {d}_{q}}{(\alpha -
		2)c_{\epsilon} {d}_{\alpha}}\right)^{\frac{1}{\alpha-q}} \\ &=
		t_{\lambda_{*}(R),R} = \left( \frac{2-q}{4c_{\epsilon}(\alpha
		- q) d_{\alpha}} \right)^{\frac{1}{\alpha-2}} = \rho_{+}(R).
	\end{align*}
	Since both $d_{\alpha}$ and $c_{\epsilon}$ are increasing
	functions of $R$, we have that $\rho_{+}(R)$ is a decreasing
	functions of $R$, while $\rho_{-}(\lambda,R)$ is an increasing
	function of $\lambda \in (0,\lambda_{*}(R))$.
	
	For all $0 < \lambda < \lambda_{*}(R)$ and $\rho_{-}(\lambda, R)
	\leq \rho \leq \rho_{+}(R)$ we have
	\begin{multline*}
		H_{\lambda,R}(t_{\lambda,R}) \leq
		H_{\lambda,R}(\rho) \leq
		H_{\lambda,R}(t_{\lambda_{*},R}) \\= \lambda {d_{q}}
		t_{\lambda_{*},R}^{q-2} + c_{\epsilon} {d_{\alpha}}
		t_{\lambda_{*},R}^{\alpha-2} 
		\leq \lambda_{*} {d_{q}} t_{\lambda_{*},R}^{q-2} +
		c_{\epsilon} {d_{\alpha}} t_{\lambda_{*},R}^{\alpha-2} \leq
		\frac{1}{4},
	\end{multline*}
	and we deduce that for all $u$ such that $\rho_{-}(\lambda,R) \leq
	\normh{u} = \rho \leq \rho_{+}(R)$
	\begin{align*}
		\Ip(u) &\geq \rho^{2} \left(\frac{1}{2}(1 - \epsilon {d_{2}})
		- H_{\lambda,R}(\rho) \right) \\
		&\geq \rho^{2} (\frac{3}{8} - \frac{1}{4}) \geq 
		\frac{\rho_{-}(\lambda, R)^{2}}{8}.
	\end{align*}
	Let us also remark that since we can assume that $c_{\epsilon}
	\geq 1$ and since $d_{\alpha} = C(N, \alpha, R) \to +\infty$ as $R
	\to +\infty$ we have that $\rho_{+}(R) \to 0$ as $R \to +\infty$.
\end{proof}

The previous Lemma, together with Lemma \ref{lem:globalmin} show that
$\Ipm$ have the geometry of the Mountain Pass Theorem. 

We now prove the existence of local minima (different from the global
ones) in a neighborhood of the origin.
\begin{lem}
	\label{lem:localmin}
	For all $R > 0$ there exists $\lambda_{**} < \lambda_{*}(R)$ such
	that for all $\lambda \in (0,\lambda_{**})$ the functionals $\Ipm$
	have local minima $v_{\pm}$ such that $\normh{v_{\pm}} \leq
	\rho_{-}(\lambda, R)$ and 
	\begin{equation*}
		0 > I(v_{\pm}) = \min_{B_{\rho_{+}(R)}} \Ipm > -R^{N}.
	\end{equation*}
\end{lem}
\begin{proof}
	Let us prove the existence of $v_{+}$, the existence of $v_{-}$ 
	follows from the same arguments applied to $\Imm$.
	
	For any $\lambda\in (0,\lambda_{*}(R))$ let
	\begin{equation*}
		B_{\rho(R)} = \settc {u \in K} {\normh{u} \leq 
		\rho(R)}, \qquad  \partial B_{\rho(R)} = \settc{u \in 
		K}{\normh{u}= \rho(R)}.
	\end{equation*}
	We know from Lemma \ref{lem:locmin} that for all $u$ with $\normh{u} = 
	\rho(R)$ we have that
	\begin{equation*}
		\Ip(u) \geq \frac{\rho_{-}(\lambda, R)^{2}}{8} \qquad \text{for
		all } u \in \partial B_{\rho(R)}.
	\end{equation*}
	Set
	\begin{equation*}
		d^{+}_{\lambda,R} = \inf \settc{ \Ip(u)}{u \in B_{\rho(R)}}.
	\end{equation*}
	Take any positive function $v \in K$ (for example $\bar{v}_{+} = 
	R\left( 1 - \frac{r}{R}\right)$) and let $t \in
	(0,\frac{\rho(R)}{\normh{v}})$ so that $tv \in B_{\rho(R)}$. We
	have that
	\begin{align*}
		\Ip(tv) &\leq \int^{R}_{0}r^{N-1} \abs{tv'}^{2} \, dr
		- \int^{R}_{0} r^{N-1} \frac{\lambda b_{0}}{q} \abs{tv^{+}}^{q} \,
		dr - \int^{R}_{0} r^{N-1} \ffp(r,tv) \, dr \\
		&\leq t^{2} \int^{R}_{0}r^{N-1}\abs{v'}^{2} \, dr - t^{q}
		\int^{R}_{0}r^{N-1}\frac{\lambda b_{0}}{q}\abs{v}^{q}\, dr < 0,
	\end{align*}
	for $t$ small enough. Consequently, $d^{+}_{\lambda} < 0$. From
	Proposition \ref{minimumcritical} it follows that
	$d_{\lambda}^{+}$ is a critical value and there exists $w_{+} \in
	K$ such that $\Ip(w_{+}) = d_{\lambda}^{+}$.

	We now estimate the value of the infimum $d_{\lambda}^{+}$. As in 
	the proof of Lemma \ref{lem:locmin} we have that
	\begin{align*}
		\Ip(u) &\geq \frac{1}{2}(1 - \epsilon {d_{2}}) \normh{u}^{2} -
		\lambda {d_{q}} \normh{u}^{q} - c_{\epsilon}
		{d_{\alpha}}\normh{u}^{\alpha}\\
		&\geq \inf_{\rho \leq \rho_{-}(\lambda,R)} \left\{
		\frac{1}{2}(1 - \epsilon {d_{2}})\rho^{2} - \lambda {d_{q}}
		\rho^{q} - c_{\epsilon} d_{\alpha} \rho^{\alpha} \right\} \\
		&\geq - \lambda {d_{q}} \rho_{-}(\lambda,R)^{q} - c_{\epsilon}
		d_{\alpha} \rho_{-}(\lambda,R)^{\alpha}.
	\end{align*}
	Recalling that $\rho_{-}(\lambda,R) \to 0$ as $\lambda \to 0$ we
	deduce that there exists $\lambda_{**} < \lambda_{*}(R)$ such that
	for all $\lambda \in (0,\lambda_{**})$
	\begin{equation*}
		d^{+}_{\lambda,R} \geq -R^{N}.
	\end{equation*}
\end{proof}

\begin{proof}[Proof of Theorem \ref{thm:ex6}]
	Let us prove the existence of three critical points of the
	functional $\Ip$ to which correspond, see Lemma
	\ref{lem:criticalAREsolutions}, three positive solutions of
	\eqref{MCE}.
	
	The first critical point $u_{+}$ is the global minimum of the
	functional $\Ip$, whose existence is given by Lemma
	\ref{lem:globalmin}. We also know that there is a function
	$\bar{u}_{\pm}$ with norm $\normh{\bar{u}_{\pm}}^{2} =
	\frac{R^{N}}{N}$ and such that $\Ip(\bar{u}_{+}) \leq -R^{N}$.
	
	This together with Lemma \ref{lem:locmin} shows that $\Ip$ has the
	geometry of the Mountain Pass Theorem
	\cite{Ambrosetti_Rabinowitz_1973} (remark that
	$\normh{\bar{u}_{\pm}} > \rho_{-}(\lambda,R)$, possibly taking
	$\lambda_{*}$ smaller). Since the Palais-Smale condition holds
	true, see Lemma \ref{lem:PS}, we immediately deduce the existence
	of a critical point $w_{+}$ at mountain pass level
	\begin{equation*}
		c = \inf_{\gamma \in \Gamma} \max_{t \in [0,1]} 
		\Ip(\gamma(t)),
	\end{equation*}
	where $\Gamma(t) = \settc{\gamma \in C([0,1])}{\gamma(0) = 0, \
	\gamma(1) = u_{+}}$. In particular $\Ip(w_{+}) \geq
	\frac{\rho_{+}(R)^{2}}{8}$.
	
	The existence of a local minimum $v_{+}$ follows from Lemma
	\ref{lem:localmin}. Moreover we have, for $\lambda \in (0,
	\lambda_{**})$ that
	\begin{equation}
		\Ip(w_{+}) \geq \frac{\rho_{+}^{2}(R)}{8}> 0 > \Ip(v_{+}) \geq
		-R^{N} > \Ip(u_{+})
	\end{equation}
	which shows that the three critical points are different and
	concludes the proof of the Theorem.
\end{proof}

\section{Existence of an additional solution}

The goal of this section is to establish the existence of an
additional critical point $v_{0}$ satisfying $I_{\lambda}(v_{0})<0$.
We find it as Mountain Pass Theorem between the two local minima
$v_{+}$ and $v_{-}$, following ideas of Ambrosetti, Garcia Azorero and
Peral in \cite{Ambrosetti_GarciaAzorero_Peral_1996}. 

We assume in this section that \ref{it:fC}, \ref{it:fat0},
\ref{it:AR}, \ref{it:fsuper}, \ref{it:asymmetry}, \ref{it:fprimo} and
\ref{it:b} hold.

The first step is to show that $v_{\pm}$, the local minima of $\Ipm$
found in Theorem \ref{thm:ex6}, are also local of the complete
functional $I_{\lambda}$. In this section we will modify the 
nonlinear term and assume that $f(r,s)$ is as follows:
\begin{equation*}
	\hat{f}(r,s) =
	\begin{cases}
	   f(r,s) &\text{if }  \abs{s} \leq R \\
	   \text{linear} &\text{if } R < \abs{s} < R+1 \\
	   0 &\text{if }  \abs{s} \geq R+1
	\end{cases}
\end{equation*}
The same argument as in section \ref{sec:variational} shows that
critical points and weak solutions involving the modified nonlinearity
are actually critical points and solutions of the original problem. We
will write $f$ instead of $\hat{f}$ for simplicity. Also all the
results of section \ref{sec:variational}, with the exception of sign
condition on the solutions in Lemma \ref{lem:criticalAREsolutions},
hold in this setting.

\begin{lem}
	\label{lem:lmI}
	Let $v_{\pm} \in \mathcal{A}$ be the local minimizers of $\Ipm$
	whose existence is given by Lemma \ref{lem:localmin}. Then
	$v_{\pm}$ are local minima also for the functional $I_{\lambda}$,
	i.e. there exists $\delta > 0$ such that
	\begin{equation*}
		I_{\lambda}(v_{\pm}) \leq I_{\lambda}(u), \qquad \forall u \in
		\mathcal{A}, \ \normh{u - v_{\pm}} \leq \delta.
	\end{equation*}
\end{lem}

\begin{proof}
	Recall that by Lemma \ref{lem:locmin} $v_{\pm}$ are local minima
	per $\Ipm$ in the ball of radius $ \rho_{-} (\lambda,R)$ and
	in the ball of radius $ \rho_{+} (R)$. Suppose the conclusion
	does not hold. Then for any $\delta > 0$, there exists $\vdp$ with
	$\normh{u_{+} - \vdp} \leq \delta$ such that
	\begin{equation}
		\label{Iv+}
		I_{\lambda}(u_{+}) > I_{\lambda}(\vdp).
	\end{equation}
	We will show that we can build, starting with $\vdp$ for $\delta$
	small enough, a function $\vdb \geq 0$ such that $\normh{\vdb}
	\leq \normh{\vdp} \leq \rho_{-}(\lambda,R) + \delta <
	\rho_{+}(R)$ and
	\begin{equation*}
		I_{\lambda}(\vdp) > I_{\lambda}(\vdb) = \Ip(\vdb),
	\end{equation*}
	a contradiction which proves the Lemma.

	We know from Lemma \ref{lem:positivity} that $v_{+}'(R) = -\mu <
	0$, and that $v_{+}(r)$ is a strictly decreasing function of $r
	\in (0,R)$. Hence for all $\epsilon \in (0, v_{+}(0))$ there is
	$r_{1} > 0$ such that $v_{+}(r) > \epsilon > 0$ for all $r \in
	(0,r_{1})$. From $v_{+}(r) \approx -\mu(r-R)$ for $r$ close to
	$R$, we also have that for $\epsilon$ small $r_{1} \approx R -
	\frac{\epsilon}{\mu}$.
	
	Take any $\epsilon > 0$ such that $3\epsilon < v_{+}(0)$ and any
	$0 < \tilde{r} < \epsilon$. For all $r \in (\tilde{r}, R)$ we have
	that
	\begin{align*}
		\abs{\vdp(r)- &v_{+}(r)} = \labs{\int_{r}^{R} (\vdp'(s)-v'_{+}(s)) \,
		ds }\\
		&\leq (R-r)^{1/2} \left( \int_{r}^{R} (\vdp'(s)-v'_{+}(s))^{2}
		\, ds \right)^{1/2} \\
		&\leq (R-r)^{1/2} \left( \int_{r}^{R} \left(\frac{s}{\tilde{r}}
		\right)^{N-1} (\vdp'(s)-v'_{+}(s))^{2} \, ds \right)^{1/2} \\
		&\leq \frac{R^{1/2}}{\tilde{r}^{(N-1)/2}} \left( \int_{0}^{R}
		s^{N-1} (\vdp'(s)-v'_{+}(s))^{2} \, ds \right)^{1/2} \leq
		\frac{R^{1/2}\delta}{\tilde{r}^{(N-1)/2}} .
	\end{align*}
	Taking $\delta$ so small so that $\frac{R^{1/2}}{\tilde{r}^{(N -
	1)/2}} \delta < \epsilon$, we have that $\vdp(r) > 0$ for all $r
	\in (\tilde{r}, r_{1})$ and $\vdp(r) > -\epsilon$ for all $r \in
	(\tilde{r}, R)$. 
	
	Since $\abs{v'_{+}(r)} \leq 1$ and $\abs{\vdp'(r)} \leq 1$ we also
	have that
	\begin{equation*}
		\vdp(r) \geq \vdp(\tilde{r}) - \epsilon \geq v_{+}(\tilde{r})
		- 2\epsilon \geq v_{+}(0) - 3\epsilon > 0 \qquad \text{for all
		} r \in (0,\tilde{r})
	\end{equation*}
	and we deduce that $\vdp(r) > 0$ for all $r \in (0,r_{1})$ and
	$\vdp(r) > -\epsilon$ for all $r \in (0, R)$.

	Assuming that $\vdp(r) < 0$ for some $r \in (r_{1}, R)$ we will
	show that for $\delta$ small there is $\vdb$ such that
	$\normh{\vdb} < \rho_{-}(\lambda,R)$, $\vdb(r) \geq 0$ for all $r
	\in (0,R)$ and $I_{\lambda}(\vdp) \geq I_{\lambda}(\vdb)$.

	To construct the new function $\vdb$ assume that for some $r_{1}
	\leq r_{2} < r_{3} \leq R$ we have that $\vdp(r_{2}) = \vdp(r_{3})
	= 0$, that $0 > \vdp(r) > -\epsilon$ for all $r \in I = (r_{2},
	r_{3})$ and that $\vdp(r) > 0$ in $(0,r_{2})$.
	
	Let also $\bar{r} \in I$ be the minimum point of $\vdp$ in $I$ and
	let $w_{1} = \abs{\vdp(\bar{r})} < \epsilon$. We also let $\tau =
	\sup \settc{r < r_{2}}{\vdp(r) > w_{1}}$. We have that $\tau > 0$
	provided $\epsilon$ is small enough.
	
	We now define
	\begin{equation*}
		\vdb(r) = 
		\begin{cases}
			\vdp(r) & r \in (0,\tau) \cup (r_{3}, R) \\
			w_{1} & r \in (\tau,\bar{r})\\
			-\vdp(r) & r \in (\bar{r}, r_{3}) 
		\end{cases}.
	\end{equation*}
	One can immediately check that
	\begin{equation*}
		\normh{\vdb}^{2} = \int_{0}^{R} r^{N-1} \vdp'(r)^{2} \, dr -
		\int_{\tau}^{\bar{r}} r^{N-1} \vdp'(r)^{2} \, dr <
		\normh{\vdp}^{2}
	\end{equation*}
	and hence
	\begin{equation*}
		\normh{\vdb} \leq \normh{\vdp} \leq \rho_{-}(\lambda,R) + 
		\delta \leq \rho_{+}(R).
	\end{equation*}
	Let us now estimate 
	\begin{align*}
		I_{\lambda}(\vdp) - I_{\lambda}(\vdb) = &\int_{\tau}^{\bar{r}}
		r^{N-1} \left(1 - \sqrt{1 - \abs{\vdp'}^{2}}\right) \, dr\\
		& + \frac{\lambda}{q} \int_{\tau}^{\bar{r}} r^{N-1} b(r)
		(w_{1}^{q} - \abs{\vdp(r)}^{q}) \, dr \\
		& + \int_{\tau}^{\bar{r}} r^{N-1} (F(r, w_{1}) - F(r,
		\vdp(r)))\, dr \\
		&+ \int_{\bar{r}}^{r_{3}} r^{N-1}
		(F(r, -\vdp(r)) - F(r,\vdp(r)))\, dr.
	\end{align*}	
	
	Since $\abs{\vdp(r)} \leq w_{1}$ for all $r \in (\tau, r_{3})$ we 
	have that the second term is positive. 
	
	Let us now analyse the other terms. We have that for all $r \in
	(r_{2}, \bar{r})$
	\begin{align*}
		F(r, w_{1}) - F(r, \vdp(r)) &= F(r, w_{1}) - F(r, -\vdp(r)) +
		F(r, -\vdp(r)) - F(r, \vdp(r)) \\
		&\geq F(r, -\vdp(r)) - F(r, \vdp(r)) \geq 
		-\abs{\vdp(r)}^{\bar{\theta}},
	\end{align*}
	where we have used assumption \ref{it:asymmetry} and the fact that
	from \ref{it:AR} follows that $f(r,s) \geq 0$ for all $(r,s) \in 
	[0,R] \times [0,R]$.
	
	Hence we have that
	\begin{equation*}
		I_{\lambda}(\vdp) - I_{\lambda}(\vdb) \geq  \frac{1}{2}
		\int_{\tau}^{\bar{r}} r^{N-1} \abs{\vdp'}^{2} \, dr
		- \int_{r_{2}}^{r_{3}} r^{N-1} \abs{\vdp}^{\bar{\theta}} \,
		dr.
	\end{equation*}
	
	Follows from the Sobolev embedding that
	\begin{align*}
		&w_{1}^{2} \leq C_{1} \int_{r_{2}}^{\bar{r}} r^{N-1}
		\abs{\vdp'}^{2} \, dr \\
		& \int_{r_{2}}^{r_{3}} r^{N-1} \abs{\vdp}^{2} \, dr \leq C_{2}
		\int_{r_{2}}^{r_{3}} r^{N-1} \abs{\vdp'}^{2} \, dr \leq
		C_{2}\rho_{+}(R)^{2}
	\end{align*}
	and hence
	\begin{align*}
		I_{\lambda}(\vdp) - I_{\lambda}(\vdb) &\geq \frac{1}{2}
		\int_{\tau}^{\bar{r}} r^{N-1} \abs{\vdp'}^{2} \, dr -
		w_{1}^{\bar{\theta}-2} \int_{r_{2}}^{r_{3}} r^{N-1}
		\abs{\vdp}^{2} \, dr \\
		&\geq \frac{1}{2} \int_{r_{2}}^{\bar{r}} r^{N-1}
		\abs{\vdp'}^{2} \, dr - \left( C_{1}\int_{r_{2}}^{\bar{r}}
		r^{N-1} \abs{\vdp'}^{2} \, dr
		\right)^{\frac{\bar{\theta}-2}{2}} C _{2} \rho_{+}(R)^{2}
	\end{align*}
	which is positive provided $\int_{r_{2}}^{\bar{r}} r^{N-1}
	\abs{\vdp'}^{2} \, dr$ is small enough since
	$\frac{\bar{\theta}-2}{2} > 1$ by assumption \ref{it:asymmetry}.
	To show that $\int_{r_{2}}^{\bar{r}} r^{N-1} \abs{\vdp'}^{2} \,
	dr$ goes to zero as $\delta$ goes to zero it is enough to remark
	that
	\begin{equation*}
		\int_{r_{2}}^{\bar{r}} r^{N-1} \abs{\vdp'}^{2} \, dr \leq
		(r_{3} -r_{2}) R^{N-1}
	\end{equation*}
	and $r_{3}-r_{2} \leq R - r_{1} \approx \frac{\epsilon}{\mu}$ and
	$\epsilon \to 0$ as $\delta \to 0$.
	
	In case $\vdb$ is not everywhere positive, we can repeat this
	argument. It is clear that after a finite number of steps we can
	stop, by setting $\vdb(r) \equiv 0$ for all $r \in (r_{3}, R)$
	provided $R - r_{3}$ is sufficiently small and we have that
	$\vdb(r) \geq 0$ for all $r \in (0,R)$ and $\normh{\vdb} \leq
	\rho(R)$ is such that
	\begin{equation*}
		I_{\lambda}(v_{+}) > I_{\lambda}(\vdb)
	\end{equation*}
	and the Lemma follows.
\end{proof}

\begin{lem}
	\label{lem:sevenMP}
	There is $\lambda_{***}$ such that for all $\lambda \in (0,
	\lambda_{***})$ there is $\gamma \in C([0,1], \mathcal{A})$ such
	that
	\begin{equation*}
		\gamma(0) = v_{+}, \qquad \gamma(1) = v_{-} \qquad \text{and} 
		\quad \max_{t \in [0,1]} I_{\lambda}(\gamma(t)) < 0.
	\end{equation*}
\end{lem}

\begin{proof}
	Let $g(t) = I_{\lambda}(tv_{+})$. Then
	\begin{multline*}
		\frac{g'(t)}{t} = \int^{R}_{0} \frac{r^{N-1}
		\abs{v_{+}'}^{2}}{\sqrt{1 - \abs{tv_{+}'}^{2}}} \, dr -
		t^{q-2} \int^{R}_{0} r^{N-1} \lambda b(r)\abs{v_{+}}^{q} \, dr
		\\ - \frac{1}{t} \int^{R}_{0} r^{N-1}f(r, tv_{+})v_{+}\, dr.
 	\end{multline*}
	We have that 
	\begin{equation*}
		\ell(t) = \int^{R}_{0} \frac{r^{N-1} \abs{v_{+}'}^{2}}{\sqrt{1
		- \abs{tv_{+}'}^{2}}} \, dr
	\end{equation*}
	is an increasing function of $t > 0$. Let us analyze the function
	\begin{equation*}
		\beta(t) = t^{q-2} \int^{R}_{0} r^{N-1} \lambda
		b(r)\abs{v_{+}}^{q} \, dr + \frac{1}{t} \int^{R}_{0}
		r^{N-1}f(r, tv_{+})v_{+}\, dr.
	\end{equation*}
	We have that $\lim_{t \to 0+} \beta(t) = \lim_{t \to +\infty}
	\beta(t) = +\infty$. We also know that $g'(1) = \ell(1) - \beta(1)
	= 0$ since $v_{+}$ is a critical point for $I_{\lambda}$.
	
	The derivative of $\beta$ is
	\begin{multline*}
		\beta'(t) = (q-2)t^{q-3} \int^{R}_{0} r^{N-1} \lambda
		b(r)\abs{v_{+}}^{q} \, dr \\ 
		- \frac{1}{t^{2}} \int^{R}_{0} r^{N-1} f(r, tv_{+})v_{+}\, dr
		+ \frac{1}{t} \int^{R}_{0} r^{N-1} f'(r, tv_{+})v_{+}^{2} \,
		dr,
	\end{multline*}
	which is negative for $t \in J = (0, \tilde{t})$, for a $\tilde{t}
	> 0$ such that $\beta'(\tilde{t}) = 0$. We deduce that $\beta$ is
	decreasing in $J$ and hence that $\ell(t) - \beta(t)$ has at most
	one zero in $J$. Let us estimate $\tilde{t}$.
	
	From $g'(1) = 0$ we deduce that 
	\begin{equation*}
		\int^{R}_{0} r^{N-1} \lambda b(r)\abs{v_{+}}^{q} \, dr =
		\int^{R}_{0} \frac{r^{N-1} \abs{v_{+}'}^{2}}{\sqrt{1 -
		\abs{v_{+}'}^{2}}} \, dr - \int^{R}_{0}
		r^{N-1}f(r, v_{+})v_{+}\, dr,
	\end{equation*}
	and hence 
	\begin{multline*}
		0 = \beta'(\tilde{t}) = (q-2) \tilde{t}^{q-3} \left(
		\int^{R}_{0} \frac{r^{N-1} \abs{v_{+}'}^{2}}{\sqrt{1 -
		\abs{v_{+}'}^{2}}} \, dr - \int^{R}_{0} r^{N-1}f(r,
		v_{+})v_{+}\, dr \right) \\
		- \frac{1}{\tilde{t}^{2}} \int^{R}_{0} r^{N-1} f(r, \tilde{t}
		v_{+})v_{+}\, dr + \frac{1}{\tilde{t}} \int^{R}_{0} r^{N-1}
		f'(r, \tilde{t}v_{+})v_{+}^{2} \, dr.
	\end{multline*}
	From \ref{it:fprimo} we deduce that for all $s$ small $f(r, s) =
	f'(r, \bar{s})s \leq c\abs{s}^{\tilde{\theta} - 1}$, and hence
	\begin{align*}
		(2 - q) & \tilde{t}^{q-3} \left( \int^{R}_{0} \frac{r^{N-1}
		\abs{v_{+}'}^{2}}{\sqrt{1 - \abs{v_{+}'}^{2}}} \, dr -
		\int^{R}_{0} r^{N-1}f(r, v_{+})v_{+}\, dr \right) \\
		&\leq (2 - q)\tilde{t}^{q-3} \left( \int^{R}_{0} \frac{r^{N-1}
		\abs{v_{+}'}^{2}}{\sqrt{1 - \abs{v_{+}'}^{2}}} \, dr -
		\int^{R}_{0} r^{N-1}f(r, v_{+})v_{+}\, dr \right) \\
		&\qquad + \frac{1}{\tilde{t}^{2}} \int^{R}_{0} r^{N-1} f(r,
		\tilde{t} v_{+})v_{+}\, dr\\
		&= \frac{1}{\tilde{t}} \int^{R}_{0} r^{N-1} f'(r,
		\tilde{t}v_{+})v_{+}^{2} \, dr \leq c\tilde{t}^{\tilde{\theta}-3}
		\int^{R}_{0} r^{N-1} \abs{v_{+}}^{\tilde{\theta}} \, dr
	\end{align*}
	and also
	\begin{align*}
		\tilde{t}^{\tilde{\theta}-q} &\geq (2-q)\frac{ \int^{R}_{0} \frac{r^{N-1}
		\abs{v_{+}'}^{2}}{\sqrt{1 - \abs{v_{+}'}^{2}}} \, dr -
		\int^{R}_{0} r^{N-1}f(r, v_{+})v_{+}\, dr }{c\int^{R}_{0}
		r^{N-1} \abs{v_{+}}^{\tilde{\theta}} \, dr} \\
		&\geq (2-q)\frac{ \int^{R}_{0} r^{N-1} \abs{v_{+}'}^{2} \, dr -
		\int^{R}_{0} r^{N-1}f(r, v_{+})v_{+}\, dr }{c\int^{R}_{0}
		r^{N-1} \abs{v_{+}}^{\tilde{\theta}} \, dr} \\
		&\geq (2-q)\frac{ \int^{R}_{0} r^{N-1} \abs{v_{+}'}^{2} \, dr
		- c\int^{R}_{0} r^{N-1} \abs{v_{+}}^{\tilde{\theta}} \, dr
		}{c\int^{R}_{0} r^{N-1} \abs{v_{+}}^{\tilde{\theta}} \, dr} \\
		&\geq \frac{(2-q)}{c} \left(\frac{ \normh{v_{+}}^{2}}{
		C(N,\tilde{\theta},R)\normh{v_{+}}^{\tilde{\theta}}} - c
		\right) \\
		&\geq \frac{(2-q)}{c} \left(\frac{1}{C(N,\tilde{\theta},R)
		\rho_{-}(\lambda, R)^{\tilde{\theta} -2 }} -
		 c \right).
	\end{align*}
	Taking $\lambda$ small we can make $\tilde{t} > 1$. This implies
	that $t = 1$ is the only zero of $\ell(t) - \beta(t)$ in $J$, that
	$g'(t) = t(\ell(t) - \beta(t))$ is negative in $(0,1)$ and we
	finally find that $I_{\lambda}(t v_{+}) < 0$ for all $t \in
	(0,1]$.
	
	The same argument also apply to show that $I_{\lambda}(tv_{-}) <
	0$ for all $t \in (0,1]$.
	
	We can now build the path $\gamma$. 

	Consider 2-dimensional plane $\Pi_{2}$ containing the
	straight lines $tv_{+}$ and $tv_{-}$ (If $v_{+}$ and $v_{-}$ are
	proportional, take any 2-dimensional plane containing them), and
	take any $v\in \Pi_{2}$ with $\normh{v} = \varepsilon_{2}$. Since
	in the 2-dimensional plane $\Pi_{2}$ all the norms are equivalent
	and $F(r,s) \geq 0$, we have that
	\begin{align*}
		\label{Iv}
		I_{\lambda}(v) &\leq \int^{R}_{0}r^{N-1}\abs{v'}^{2} \, dr -
		\int^{R}_{0} r^{N-1} \frac{\lambda b_{0}}{q}\abs{v}^{q} \,
		dr - \int^{R}_{0} r^{N-1} F(r,v) \, dr \\
		&\leq (\varepsilon_{2})^{2} - \frac{\lambda
		b_{0}}{q}(c^{*})^{q}(\varepsilon_{2})^{q} < 0.
	\end{align*}
	Consider the path $\gamma$ obtained gluing together the segments
	$tv_{-}$ for $t \in (\varepsilon_{2} \normh{v_{-}}^{-1}, 1)$,
	$tv_{+}$, $t \in (\varepsilon_{2} \normh{v^{+}}^{-1}, 1)$ and the
	arc on the circle in $\prod_{2}$ $\settc{v \in \prod_{2}}{\normh{v}
	= \varepsilon_{2}}$. From the above, we know
	\begin{equation*}
		\max_{v\in\bar{\gamma} } I_{\lambda}(v)<0.
	\end{equation*}
\end{proof}

\begin{proof}[Proof of Theorem \ref{thm:ex7}]
	We know from Lemma \ref{lem:lmI} that $v_{\pm}$ are local minima
	for the functional $I_{\lambda}$. We can also assume $v_{\pm}$ are
	isolated local minima. We let
	\begin{equation*}
		b_{\lambda} = \inf_{\gamma \in \Gamma} \max_{t \in [0,1]} 
		I_{\lambda}(\gamma(t))
	\end{equation*}
	where 
	\begin{equation*}
		\Gamma = \settc{\gamma \in C([0,1], \mathcal{A})}{\gamma(0) = 
		v_{-}, \ \gamma(1) = v_{+}, \ \normh{\gamma(t)} \leq \rho(R)}.
	\end{equation*}
	We know from Lemma \ref{lem:sevenMP} that $b_{\lambda} < 0$
	and that $I_{\lambda}$ satisfies the (PS) condition. 
	Then we can apply the Mountain Pass Theorem and find the seventh 
	solution of our problem.
\end{proof}

\section{Multiplicity result involving gradient term}
\label{sec:grad}

In this section, we will prove the results on multiplicity of
solutions for the mean curvature equation with nonlinearity depending
also on the gradient \eqref{GMCE}. We remark that this problem does
not have a variational structure. 

We will follow ideas introduced by De Figueiredo, Girardi, and Matzeu
\cite{DeFigueiredo_Girardi_Matzeu_2004} to deal with semilinear
elliptic equations involving the Laplacian. If we look for radial
solutions of \eqref{GMCE}, we are lead to consider the one-dimensional
mixed boundary problem
\begin{equation}
	\label{RGMCE}
	\begin{cases}
		-\left(\frac{r^{N-1}u'}{\sqrt{1 - \abs{u'}^{2}}} \right)' =
		\lambda b(r) r^{N-1} \abs{u}^{q-2} u + r^{N-1} g(r, u,
		\abs{u'}) &\text{in } (0,R) \\
		u'(0)=u(R)=0 
	\end{cases}
\end{equation}

In order to apply the critical point Theorem, we freeze the derivative
term and study a modified problem which does not depend on the
derivative term. Namely we consider, for each $\omega \in C^{1}[0,R]
\cap K$, the following problem
\begin{equation}
	\label{RGMCEw}
	\begin{cases}
		-\left(\frac{r^{N-1}u'}{\sqrt{1 - \abs{u'}^{2}}} \right)' =
		\lambda b(r) r^{N-1} \abs{u}^{q-2} u + r^{N-1} g(r, u,
		 \abs{\omega' }) &\text{in } (0,R) \\
		u'(0)=u(R)=0 
	\end{cases}.
\end{equation}
It is clear that for all $\omega \in C^{1}[0,R] \cap K$ the function
$f(r,s) \equiv g(r, s, \abs{\omega'(r)})$ satisfies assumptions
\ref{it:fC}, \ref{it:fat0}, \ref{it:AR} and \ref{it:fsuper} if $g$
satisfies assumptions \ref{it:fC*}, \ref{it:fat0*}, \ref{it:AR*} and
\ref{it:fsuper*}.

Then we immediately deduce
\begin{lem}
	\label{lem:ex4grad}
	Let $g$ satisfy \ref{it:fC*}, \ref{it:fat0*}, \ref{it:AR*} and
	\ref{it:fsuper*} and assume \ref{it:b} holds. Let $\omega \in
	C^{1}([0,R]) \cap K$. Then
	\begin{itemize}
		\item For all $\lambda > 0$ \eqref{RGMCEw} has one positive
		solution $u_{\omega}^{+}$ and one negative solution
		$u_{\omega}^{-}$, global minima of the corresponding
		functionals;
	
		\item There is $\lambda_{*}(R)$ such that for all $\lambda \in
		(0, \lambda_{*}(R))$ \eqref{RGMCEw} has one additional
		positive solution $v_{\omega}^{+}$ and one additional negative
		solution $v_{\omega}^{-}$ of mountain pass type.
	
	\end{itemize}
\end{lem}

Let us remark that here, as in section \ref{sec:variational}, we
modify the function $g(r,s, \abs{\omega'(r)})$ in the following
way
\begin{equation*}
	\gp{w}(r, s) =
	\begin{cases}
	   g(r, s, \abs{\omega'(r)}) &\text{if } 0 \leq s \leq R
	   \\
	   -g(r, R, \abs{\omega'(r)})(s - R - 1) &\text{if } R < s <
	   R+1 \\
	   0 &\text{if } s < 0 \text{ or } s \geq R+1
	\end{cases}
\end{equation*}
and similarly for the function $\gm{\omega}(r,s)$. It is immediate to
observe that $\gp{\omega}$ satisfies assumptions \ref{it:fC},
\ref{it:fat0}, \ref{it:AR} and \ref{it:fsuper} if $g$ satisfies
assumptions \ref{it:fC*}, \ref{it:fat0*}, \ref{it:AR*} and
\ref{it:fsuper*}.

For simplicity we deal here only the existence of positive solutions,
similar reasoning work for the existence of negative solutions.

We let 
\begin{multline*}
	I^{+}_{\omega,\lambda} (u) = \int^{R}_{0} r^{N-1}(1 - \sqrt{1 -
	\abs{u'}^{2}}) \, dr \\
	- \frac{\lambda }{q} \int^{R}_{0} b(r) r^{N-1} \abs{u^{+}}^{q} \,
	dr - \int^{R}_{0} r^{N-1} G^{+}_{\omega}(r, u)\, dr,
\end{multline*}
where 
\begin{equation*}
	G^{+}_{\omega}(r, s) = \int_{0}^{s} \gp{w}(r, t) \, dt.
\end{equation*}

We now give the following Lemma of the lower bound for the solutions
obtained by Lemma \ref{lem:ex4grad}. 

\begin{lem}
	\label{inf}
	Assume that $\omega \in C^{1}[0,R] \cap K$. Then there exists a
	positive constant $\varrho_{1}(R)$, independent of $\omega$, such that
	\begin{equation*}
		\normh{v^{+}_\omega} \geq \varrho_{1}(R),
	\end{equation*}
	where $v^{+}_\omega$ is the (positive) solution of mountain pass
	type whose existence is stated in Lemma \ref{lem:ex4grad}.
\end{lem}

\begin{proof}
	Follows from \ref{it:fat0} (which is a consequence of
	\ref{it:fat0*}) that there is a constant $\tilde{c} > 0$ which
	does not depend upon $\omega$ such that
	\begin{equation}
		\label{f+-}
		\abs{\gp{\omega}(r, s)} \leq \tilde{c} \abs{s}, \qquad \forall
		(r, s) \in [0,+\infty) \times \mathbb{R}.
	\end{equation}
	and also that for all $\epsilon > 0$ and $\alpha \in (2, 
	\frac{2N}{N-2})$ there is a constant 
	$C_{\epsilon}$ (which does not depend on $\omega$) such that for 
	\begin{equation*}
		0 \leq \gp{\omega}(r, s) s \leq \epsilon \abs{s}^{2} +
		C_{\epsilon} \abs{s}^{\alpha}.
	\end{equation*}
	We also remark that 
	\begin{equation*}
		\frac{s^{2}}{\sqrt{1-s^{2}}} - q\left(1 - \sqrt{1 - 
		s^{2}}\right) \geq \left( 1 - \frac{q}{2} \right) s^{2} 
		\qquad \text{for all } s \in (-1,1).
	\end{equation*}
	
	Since $v^{+}_{\omega}$ is a weak solution, it holds (where, for
	simplicity of notation, we let $v = v^{+}_{\omega}$)
	\begin{equation*}
		\int^{R}_{0} \frac{r^{N-1} \abs{v'}^{2}}{\sqrt{1 -
		\abs{v'}^{2}}} \, dr = \int^{R}_{0} r^{N-1} \lambda b(r)
		\abs{v}^{q} \, dr + \int^{R}_{0}r^{N-1} \gp{\omega}(r, v) v \,
		dr,
	\end{equation*}
	while from the fact that $v$ is a mountain pass critical point we
	deduce that
	\begin{multline*}
		I^{+}_{\omega, \lambda} (v) = \int^{R}_{0}
		r^{N-1}(1 - \sqrt{1 - \abs{v'}^{2}}) \, dr \\
		- \frac{\lambda }{q} \int^{R}_{0} r^{N-1} b(r) \abs{v}^{q} \, dr -
		\int^{R}_{0} r^{N-1} G^{+}_{\omega}(r, v) \, dr > 0.
	\end{multline*}
	Using the above, we have that
	\begin{multline*}
		\int^{R}_{0} r^{N-1}(1 - \sqrt{1 - \abs{v'}^{2}}) \, dr -
		\int^{R}_{0} r^{N-1} G^{+}_{\omega}(r, v) \, dr \\
		> \frac{\lambda }{q} \int^{R}_{0} r^{N-1} b(r) \abs{v}^{q} \, dr =
		\frac{1}{q} \int^{R}_{0} \frac{r^{N-1} \abs{v'}^{2}}{\sqrt{1 -
		\abs{v'}^{2}}} \, dr - \frac{1}{q} \int^{R}_{0}r^{N-1}
		\gp{\omega}(r, v) v \, dr,
	\end{multline*}
	and hence
	\begin{align*}
		\left( 1 - \frac{q}{2} \right) &\int^{R}_{0} r^{N-1} 
		\abs{v'}^{2} \, dr \\
		&\leq \int^{R}_{0} \frac{r^{N-1} \abs{v'}^{2}}{\sqrt{1 -
		\abs{v'}^{2}}} \, dr - q \int^{R}_{0} r^{N-1}(1 - \sqrt{1 -
		\abs{v'}^{2}}) \, dr \\
		&< \int^{R}_{0}r^{N-1} \gp{\omega}(r, v) v \, dr - q
		\int^{R}_{0} r^{N-1} G^{+}_{\omega}(r, v) \, dr \\
		&\leq \int^{R}_{0}r^{N-1} \left( \epsilon \abs{v}^{2} +
		C_{\epsilon} \abs{v}^{\alpha}\right) \, dr.
	\end{align*}
	Taking $\epsilon$ small enough so that $1 -\frac{q}{2} -
	\epsilon C(N,2,R) > 0$ we find that
	\begin{equation*}
		\left( 1 - \frac{q}{2} - \epsilon C(N,2,R) \right)
		\normh{v}^{2} \leq C_{\epsilon} C(N,\alpha,R)
		\normh{v}^{\alpha}
	\end{equation*}
	and the Lemma follows.
\end{proof}

A similar result holds for the global minimum of  
functional $I^{+}_{\omega, \lambda}$.
\begin{lem}
	\label{inf*}
	Assume that $\omega \in C^{1}[0,R] \cap K$. Then there exists a
	positive constant $\varrho_{2}(R)$, independent of $\omega$,
	such that
	\begin{equation*}
		\normh{u_\omega^{+}} \geq \varrho_{2}(R),
	\end{equation*}
	where $u_\omega^{+}$ is the (positive) global minimum whose
	existence is stated in Lemma \ref{lem:globalmin}.
\end{lem}

\begin{proof}
	It is enough to remark that, according to Lemma
	\ref{lem:globalmin}, the global minimum $u_{\omega}^{+}$ is such
	that $I^{+}_{\omega, \lambda}(u_{\omega}^{+}) \leq -R^{N}$.
	
	Suppose that there exist a sequence of functions $\omega_{n} \in
	C^{1}[0,R] \cap K$ such that the corresponding global minima
	$\{u_{\omega_{n}}\}$ are such that $\lim_{n \to +\infty}
	\normh{u_{\omega_{n}}} = 0$. By the continuity of $I^{+}_{\omega, 
	\lambda}$ we immediately deduce that 
	\begin{equation*}
		- R^{N} \geq \lim_{n \to +\infty} I^{+}_{\omega, 
		\lambda}(u_{\omega_{n}}) = 0,
	\end{equation*}
	a contradiction which proves the Lemma.
\end{proof}

\begin{proof}[Proof of Theorem \ref{thm:thg}]
	We now consider the following problem
	\begin{equation}
		\label{eq:induc}
		\tag{$MP_{\omega}^{+}$}
		\begin{cases}
			-\left( \frac{r^{N-1}u'}{\sqrt{1 - \abs{u'}^{2}}}\right)'
			= \lambda b(r) r^{N-1} \abs{u}^{q-2} u^{+} + r^{N-1}
			\gp{\omega}(r, u) &\text{in }(0, R), \\
			u'(0) = u(R) = 0.
		\end{cases}
	\end{equation}
	Choose any $\omega_0 \in C^{1}[0,R] \cap K$ and define by
	induction a sequence $v_{n}$, where $v_{n}$ is a positive mountain
	pass solution of problem \eqref{eq:induc} with $\omega =
	\omega_{n-1}$ (if there is more than one such solution, choose any
	one) and, for $n \in \mathbb{N}$, $\omega_{n} = v_{n-1}$.
	
	In such a way we construct a sequence of positive mountain pass
	solutions $\{v_n\} \subset C^{1}[0,R] \cap K$. Moreover,
	$I^{+}_{\omega_{n-1}, \lambda} (v_{n}) > 0$. 
	
	Similarly we can define by induction a sequence $u_{n}$, where
	$u_{n}$ is a global minimum solution of problem \eqref{eq:induc}
	with $\omega = \omega_{n-1}$ (if there is more than one such
	solution, choose any one) and, for $n \in \mathbb{N}$, $\omega_{n}
	= u_{n-1}$.

	We will show that $\{u_{n}\}$ and $\{v_{n}\}$ are Cauchy
	sequences. We only give the proof for $\{u_{n}\}$.

	We deduce from the fact that $u_{n}$ and $u_{n+1}$ are positive
	weak solutions of \eqref{eq:induc} and recalling the definition of
	$\gp{\omega}$ we have that
	\begin{multline}
		\label{q1}
		\int_{0}^{R} \frac{r^{N-1} u_{n}'(u_{n+1} - u_{n})'}{\sqrt{1 -
		\abs{u'_{n}}^{2}}} \, dr \\
		= \int_{0}^{R} r^{N-1} \left(\lambda b(r) u_{n }^{q - 1} +
		g(r, u_{n}, \abs{u'_{n-1}}) \right)(u_{n+1} - u_{n})\,
		dr
	\end{multline}
	and
	\begin{multline}
		\label{q2}
		\int_{0}^{R} \frac{r^{N-1} u_{n+1}'(u_{n+1} - u_{n})'}{\sqrt{1
		- \abs{u'_{n+1}}^{2}}} \, dr \\
		= \int_{0}^{R} r^{N-1} \left(\lambda b(r) u_{n+1}^{q - 1} +
		g(r, u_{n+1}, \abs{u'_{n}}) \right)(u_{n+1} - u_{n}) \,
		dr.
	\end{multline}
	By Lemma \ref{C1}, we know $\abs{u_n'(r)} < 1$. Then Mean Value
	Theorem shows that
	\begin{multline*}
		\int_{0}^{R}r^{N-1} \frac{\abs{(u_{n+1} - u_{n})'}^{2}}{(1 -
		\abs{\psi'}^{2})^{\frac{3}{2}}} \, dr = \int_{0}^{R} r^{N-1}
		\lambda b(r) (u_{n+1}^{q-1} - u_{n}^{q-1} ) (u_{n+1} - u_{n})
		\, dr \\
		+ \int_{0}^{R} r^{N-1} \left(g(r, u_{n+1}, \abs{u'_{n}}
		) - g(r, u_{n}, \abs{u'_{n-1}} ) \right)(u_{n+1} -
		u_{n}) \, dr ,
	\end{multline*}
	where $\psi' = u'_{n} + \theta(u'_{n+1} - u'_{n})$ for some
	$\theta \in (0,1)$. As a consequence we have
	\begin{align*}
		\int_{0}^{R} r^{N-1} &\abs{(u_{n+1} - u_{n})'}^{2} \, dr \leq
		\int_{0}^{R} r^{N-1} \lambda b(r) ( u_{n+1}^{q-1} -
		u_{n}^{q-1} )(u_{n+1} - u_{n}) \, dr \\
		&\qquad + \int_{0}^{R} r^{N-1} \left( g(r, u_{n+1},
		\abs{u'_{n}} ) - g(r, u_{n}, \abs{u'_{n}})
		\right)(u_{n+1} - u_{n})\, dr \\
		&\qquad + \int_{0}^{R} r^{N-1} \left( g(r, u_{n},
		\abs{u'_{n}}) - g(r, u_{n}, \abs{u'_{n-1}})
		\right)(u_{n+1} - u_{n})\, dr .
	\end{align*}
	From $\abs{x^{q-1} - y^{q-1}} \leq q(\abs{x} + \abs{y})^{q-2}
	\abs{x - y}$ and using assumption \ref{it:lip} we deduce that
	\begin{multline*}
		\normh{u_{n+1} - u_{n}}^{2} \leq \lambda q b_{1} \int_{0}^{R}
		r^{N-1} ( u_{n+1} + u_{n} )^{q-2} \abs{ u_{n+1} - u_{n}}^{2}
		\, dr \\
		+ L_{1} \int_{0}^{R} r^{N-1} \abs{u_{n+1} - u_{n}}^{2} \, dr 
		\\
		+ L_{2} \int_{0}^{R} r^{N-1} \labs{ \abs{u_{n}'} -
		\abs{u_{n-1}'} } \abs{u_{n+1} - u_{n}} \, dr .
	\end{multline*}
	Recalling that $0 \leq u_{n} \leq R$ and that $u'_{n} \leq 0$ (see
	Lemma \ref{lem:positivity}) we get 
	\begin{align*}
		&\normh{u_{n+1} - u_{n}}^{2} \leq (\lambda (2R)^{q-2} q b_{1} +
		L_{1}) \int_{0}^{R} r^{N-1} \abs{u_{n+1} - u_{n}}^{2} \, dr \\
		&\qquad + L_{2}\left(\int_{0}^{R} r^{N-1} \abs{u'_{n}
		-u'_{n-1}}^{2} \, dr \right)^{\frac{1}{2}} \left( \int_{0}^{R}
		r^{N-1} \abs{u_{n+1} - u_{n}}^{2} \, dr
		\right)^{\frac{1}{2}}\\
		&\quad\leq (\lambda (2R)^{q-2}q b_{1} + L_{1} ) C_{2}
		\normh{u_{n+1} - u_{n}}^{2} + \sqrt{C_{2}} L_{2} \normh{
		u_{n}- u_{n-1}}\normh{ u_{n+1}-u_{n}},
	\end{align*}
	where $C_{2} = C(N, 2, R)$ is the Sobolev embedding constant. 
	Let
	$0 < \bar{\lambda} < \lambda_{*}(R)$ be small enough so that
	\begin{equation*}
		\lambda (2R)^{q-2}q b_{1}  C_{2} < \frac{1}{4}
	\end{equation*}
	for all $\lambda\in (0,\bar{\lambda})$. Then follows from assumption
	\eqref{eq:lipconst} that
	\begin{equation*}
		(\lambda (2R)^{q-2}q b_{1} + L_{1}) C_{2} < \frac{1}{2} \qquad
		\hbox{and} \qquad L_{2} \sqrt{C_{2}} < \frac{1}{2}
	\end{equation*}
	and
	\begin{equation*}
		\normh{u_{n+1}-u_n}\leq \frac{L_{2}\sqrt{C_{2}}}{1-(\lambda
		b_{1} (2R)^{q-2}q + L_{1})C_{2}}
		\normh{u_{n}-u_{n-1}}:=k\normh{u_{n}-u_{n-1}},
	\end{equation*}
	with $k<1$ and $\{u_n\}$ is a Cauchy sequence. Then $u_n$
	converges strongly to some $u \in \mathcal{A}$. From Lemma
	\ref{inf}, we know $\normh{u} \geq \varrho_{1}(R)$, hence it is a
	nontrivial solution.
\end{proof}

\end{document}